\documentclass[12pt]{amsart}
\usepackage[top=35mm, bottom=32mm, left=30mm, right=30mm]{geometry}
\usepackage{amsmath,amssymb,amsthm,mathtools}
\usepackage{enumitem}
\usepackage{microtype}
\usepackage[hidelinks]{hyperref}
\usepackage{mathrsfs}
\usepackage{bm}

\usepackage{aliascnt}

\newtheorem{theorem}{Theorem}[section]
\newaliascnt{proposition}{theorem}
\newtheorem{proposition}[proposition]{Proposition}
\aliascntresetthe{proposition}
\newaliascnt{lemma}{theorem}
\newtheorem{lemma}[lemma]{Lemma}
\aliascntresetthe{lemma}
\newaliascnt{corollary}{theorem}
\newtheorem{corollary}[corollary]{Corollary}
\aliascntresetthe{corollary}
\newaliascnt{claim}{theorem}

\aliascntresetthe{claim}
\theoremstyle{definition}
\newaliascnt{definition}{theorem}
\newtheorem{definition}[definition]{Definition}
\aliascntresetthe{definition}
\newaliascnt{remark}{theorem}
\newtheorem{remark}[remark]{Remark}
\aliascntresetthe{remark}

\usepackage[nameinlink,capitalise]{cleveref}
\crefname{theorem}{Theorem}{Theorems}
\crefname{proposition}{Proposition}{Propositions}
\crefname{lemma}{Lemma}{Lemmas}
\crefname{corollary}{Corollary}{Corollaries}
\crefname{claim}{Claim}{Claims}
\crefname{definition}{Definition}{Definitions}
\crefname{remark}{Remark}{Remarks}

\newcommand{\Haar}{m_G}
\newcommand{\vol}{\operatorname{vol}}
\newcommand{\Prob}{\operatorname{Prob}}
\newcommand{\Map}{\operatorname{Map}}
\newcommand{\Mapavg}{\operatorname{Map}_{\mathrm{avg}}}
\newcommand{\Sep}{\operatorname{Sep}}
\newcommand{\PPP}{\operatorname{PPP}}
\newcommand{\cN}{\mathcal N}
\newcommand{\cB}{\mathcal B}
\newcommand{\cO}{\mathcal O}
\newcommand{\1}{\mathbf 1}
\newcommand{\eps}{\varepsilon}
\newcommand{\dbar}{\overline{\vol}}
\newcommand{\dd}{\,d}
\newcommand{\norm}[1]{\left\lVert#1\right\rVert}
\newcommand{\abs}[1]{\left\lvert#1\right\rvert}

\title[Poisson actions and completely positive sofic entropy]
{Poisson actions of noncompact locally compact sofic groups
	have completely positive sofic entropy}

\author{Zhuowei Liu}
\date{\today}

\address[Z. Liu]{School of Mathematics (Zhuhai), Sun Yat-sen University,
	Zhuhai, Guangdong, 519000, P.R. China}
\email{liuzhw55@mail2.sysu.edu.cn}
\begin{document}

\begin{abstract}
	We prove that every Poisson action of positive intensity of a
	noncompact locally compact second countable sofic group has
	completely positive measure sofic entropy with respect to every
	locally compact sofic approximation. The proof constructs random microstates for arbitrary measurable
	factors by completing local Poisson configurations with an independent
	exterior Poisson process. Probability estimates and a packing argument
	then show that every nontrivial factor has positive entropy. We also give a compact-group example showing that the noncompactness assumption is essential.
\end{abstract}

\maketitle

\section{Introduction}

Completely positive entropy is stronger than positive entropy.
A probability-measure-preserving action has completely positive entropy when
each nontrivial factor has positive entropy. Thus the action and all its nontrivial factors have positive entropy.
In the classical case of a single transformation, completely
positive entropy is equivalent, by the Rokhlin--Sinai theorem, to the
$K$-property and is closely tied to strong mixing \cite{Glasner}. For countable amenable group actions, it also has strong mixing consequences
\cite{RudolphWeiss}. Bernoulli actions are the basic examples. The Ornstein--Weiss theory
implies that Bernoulli actions of countably infinite amenable groups have
completely positive entropy \cite{OrnsteinWeiss}.

Bowen introduced measure entropy for actions of countable sofic
groups by counting finite models associated with a sofic
approximation \cite{BowenSofic}. Kerr and Li developed an
operator-algebraic approach that also defines topological entropy
and establishes the variational principle \cite{KerrLi1}.
For amenable groups, these invariants agree with classical entropy
\cite{B12,KerrLi2}; see \cite{B20} for further background.
For general sofic groups, entropy may depend on the chosen
approximation. Nevertheless, Kerr proved that Bernoulli actions
have completely positive entropy with respect to every sofic
approximation \cite{KerrCPE}. His proof obtains positive entropy
bounds directly for nontrivial factors, without requiring those
factors to be Bernoulli. Austin and Burton subsequently studied
the relation between this result and uniform mixing
\cite{AustinBurton}.

Entropy theory for locally compact second countable groups requires a
different kind of finite model. Singh introduced an operator-algebraic theory
of locally compact sofic entropy and proved a spatial formula on compact
models \cite[Chapter~3]{SinghThesis}. Bowen and Burton later introduced soficity for locally compact
groups using finite-volume local $G$-spaces
\cite{BowenBurton}, and Bowen developed the corresponding entropy
theory, including measure-conjugacy invariance and the variational
principle \cite{BowenLCEntropy}. In this setting, one studies approximately equivariant maps from finite-volume spaces where the group action is defined only locally. This local geometry is
exactly what makes the Poisson case both natural and technically different
from the discrete Bernoulli case.

Let $G$ be a locally compact sofic group and
$\Haar$ a Haar measure on $G$. Let $\cN(G)$ be the space of locally finite
integer-valued Radon measures on $G$, and, for $\alpha>0$, let
\[
 \mu_\alpha:=\PPP(\alpha\Haar)
\]
be the law of the Poisson point process of intensity $\alpha\Haar$. We use the
left action by right shifts
\[
 (g\omega)(A)=\omega(Ag),\qquad g\in G.
\]

Poisson actions are locally compact analogues of Bernoulli shifts.
Their independence comes from restrictions to disjoint regions
of the group. For a countable discrete group equipped with
counting Haar measure, the coordinate map
\[
\omega\longmapsto
\bigl(\omega(\{g\})\bigr)_{g\in G}
\]
identifies the Poisson action, up to inversion of the group
coordinates, with the Bernoulli shift having base law
$\operatorname{Pois}(\alpha)$. Thus Kerr's theorem settles the
discrete case. For nondiscrete groups, individual group elements
have Haar measure zero, and the construction instead uses
Poisson configurations in local charts.

For noncompact nondiscrete locally compact sofic groups,
Singh's computation gives infinite measure sofic entropy for
Poisson actions of positive intensity
\cite[Theorem~5.2.5]{SinghThesis}. The role of noncompactness in
our argument is explained in Section~9.
Infinite entropy of the full action does not imply complete
positivity, since a nontrivial factor may have zero entropy.
Bowen therefore asked whether Poisson actions have completely
positive entropy with respect to every locally compact sofic
approximation \cite[Question~(4)]{BowenLCEntropy}.

In this paper, we answer Bowen's question for every noncompact locally
compact second countable sofic group. The main obstacle is that pulling a
Poisson process back through a local chart produces a truncated configuration lying in a $\mu_\alpha$-null set, where an
arbitrary measurable factor map is not controlled. We solve this problem by adding an independent Poisson process outside
the local chart. The resulting configuration has law $\mu_\alpha$ at
every base point and can therefore be passed through any measurable
factor map. The following is the main result.

\begin{theorem}\label{thm:main}
Let $G$ be a noncompact locally compact second countable sofic group, and
let $\Sigma=(M_i)_{i\geq1}$ be a locally compact sofic approximation of
$G$. For $\alpha>0$, let $\mu_\alpha:=\PPP(\alpha\Haar)$
be the law of the Poisson point process on $G$ with intensity measure
$\alpha\Haar$. Then the Poisson action
	\[
	G\curvearrowright(\cN(G),\mu_\alpha),
	\qquad
	(g\omega)(A)=\omega(Ag),
	\]
	has completely positive measure sofic entropy with respect to $\Sigma$.
	In other words, every nontrivial measurable factor
	\[
	\pi:(\cN(G),\mu_\alpha)\longrightarrow(Z,\zeta)
	\]
	satisfies
	\[
	h^{\mathrm{meas}}_\Sigma
	\bigl(G\curvearrowright(Z,\zeta)\bigr)>0.
	\]
\end{theorem}

The exact marginal law of these configurations gives random models
for every measurable factor. Local equivariance and Poisson concentration,
followed by a packing argument, produce exponentially many separated factor
microstates. We also prove the equivalence of the averaged and pointwise
spatial entropy formulas and compare them with Singh's entropy.
Noncompactness is essential: it gives
\[
\vol_{M_i}(M_i)\longrightarrow\infty,
\]
while for a nontrivial compact group the total number of Poisson points
defines a nontrivial invariant factor.

\medskip
\noindent\textbf{Organization of the paper.}
Section~2 recalls the spatial entropy definition and proves the equivalence of
averaged and pointwise equivariance. Section~3 collects the geometric properties of local $G$-space models. Section~4 constructs the completed configurations and proves their exact
marginal and local equivariance properties.
Section~5 proves exponential concentration for local Poisson observables.
Sections~6--7 pass the construction to arbitrary measurable factors and prove
the separation. Section~8 proves  packing estimates and the main theorem.
Section~9 gives an example showing that the theorem fails for compact groups.
Appendix~A proves the invariance statements used above and
compares the spatial formula with Singh's entropy.

\section{Spatial measure sofic entropy on compact models}\label{sec:compact-entropy}

\subsection{Standard probability-measure-preserving actions}
\begin{definition}[Standard pmp action]
	Let $G$ be a locally compact second countable group.
	A standard probability-measure-preserving (pmp) action of $G$ is a
	Borel action
	\[
	G\times X\to X,\qquad (g,x)\mapsto gx,
	\]
	on a standard probability space $(X,\mu)$ such that
$\mu(gA)=\mu(A)$
	for every $g\in G$ and every Borel set $A\subset X$.	
	In other words, the probability measure $\mu$ is invariant under the
	action of $G$.
\end{definition}
We use the abbreviation lcsc to mean locally compact second
countable.
\begin{definition}[Measurable factor]
	Let
	$G\curvearrowright(X,\mu)$,
	$G\curvearrowright(Y,\nu)$
	be standard pmp actions. A measurable \emph{factor map} is a measurable map
	$\pi:X\to Y$
	such that
$\pi_*\mu=\nu$ (where $\pi_*\mu$ is the measure on
$Y$ defined by $\pi_*\mu(A)=\mu (\pi^{-1}A)$)
	and
$\pi(gx)=g\pi(x)$
	for every $g\in G$ and for $\mu$-almost every $x\in X$; a \emph{measure conjugacy} is a factor map for which there exist invariant
	conull Borel sets $X_0\subseteq X$ and $Y_0\subseteq Y$ such that
	$\pi|_{X_0}:X_0\to Y_0$ is a Borel bijection with Borel inverse and
	intertwines the two actions on these conull sets.

\end{definition}

\begin{definition}
A \emph{compact model} of a standard pmp action $G\curvearrowright(Z,\zeta)$
consists of a compact metrizable space $Y$, a jointly continuous action
$G\curvearrowright Y$, an invariant Borel probability measure $\nu$, and a
measure conjugacy between $(Z,\zeta)$ and $(Y,\nu)$.
\end{definition}

\begin{proposition}[Existence of compact models]\label{prop:compact-model}\cite{Var63}
Every standard pmp action of an lcsc group admits a compact metrizable
continuous model.
\end{proposition}

We work throughout with compact metrizable models of the factors under
consideration. The Poisson configuration space itself is noncompact and is
used only as the domain of the factor map and as a space in the
random construction.

\subsection{Local \texorpdfstring{$G$}{G}-spaces and sofic approximations}

We recall the local-$G$-space definition of Bowen and Burton
\cite{BowenBurton}. Let $G$ be an lcsc group with identity $e$.

\begin{definition}\label{defn}
Let $M$ be a locally compact second countable Hausdorff space.
A continuous map
$\vartheta:\operatorname{dom}(\vartheta)\to M$, where
$\operatorname{dom}(\vartheta)\subseteq G\times M$ is open,
is a \emph{partial left action} of $G$ if, for every $p\in M$,
\begin{enumerate}
\item $(e,p)\in\operatorname{dom}(\vartheta)$ and
$\vartheta(e,p)=p$;
\item if $(g,p)\in\operatorname{dom}(\vartheta)$, then
$(g^{-1},\vartheta(g,p))\in\operatorname{dom}(\vartheta)$ and
$\vartheta(g^{-1},\vartheta(g,p))=p$;
\item if $(h,p)$, $(g,\vartheta(h,p))$, and $(gh,p)$ belong to
$\operatorname{dom}(\vartheta)$, then
$\vartheta(gh,p)=\vartheta(g,\vartheta(h,p))$.
\end{enumerate}
The partial action is \emph{homogeneous} if, in addition,
\begin{enumerate}[resume]
\item for every $p\in M$, there is an open neighborhood $V_p$ of $e$ such
that $V_p\times\{p\}\subseteq\operatorname{dom}(\vartheta)$ and
$\vartheta(\cdot,p)$ restricts to a homeomorphism from $V_p$ onto an open
neighborhood of $p$ in $M$.
\end{enumerate}
A pair $(M,\vartheta)$ is a \emph{local left $G$-space} if $\vartheta$ is a
homogeneous partial left action of $G$ on $M$.
\end{definition}

We usually suppress $\vartheta$ and write $g.p$ for $\vartheta(g,p)$. If
$L\subseteq M$ and $\{g\}\times L\subseteq\operatorname{dom}(\vartheta)$, we
write $g.L:=\{g.p:p\in L\}$. In the same way, if $V\subseteq G$ and
$V\times\{p\}\subseteq\operatorname{dom}(\vartheta)$, we write
$V.p:=\{g.p:g\in V\}$.

\begin{definition}\label{def:chart}
Let $(M,\vartheta)$ be a local $G$-space and let $p\in M$. A
\emph{chart centered at $p$} is a homeomorphism
$f_p:\operatorname{dom}(f_p)\to\operatorname{rng}(f_p)$ such that
\begin{enumerate}
\item $\operatorname{dom}(f_p)$ is an open neighborhood of $p$ in $M$;
\item $\operatorname{rng}(f_p)$ is an open neighborhood of $e$ in $G$;
\item $f_p(g.p)=g$ for every $g\in\operatorname{rng}(f_p)$.
\end{enumerate}
Condition~(4) of \cref{defn} gives such a chart at every point.
\end{definition}

The next two results from \cite{BowenBurton} provide the canonical measure on
a local $G$-space and its local invariance.

\begin{proposition}
Let $(M,\vartheta)$ be a local $G$-space, and fix a right Haar measure $\Haar$
on $G$. There is a unique Radon measure $\vol_M$ on $M$ such that, whenever
$f_p$ is a chart centered at $p$ and $L\subseteq\operatorname{dom}(f_p)$ is
Borel,
\[
 \vol_M(L)=\Haar\bigl(\{g\in\operatorname{rng}(f_p):g.p\in L\}\bigr)
 =\Haar(f_p(L)).
\]
When the underlying local $G$-space is clear, we write $\vol$ for $\vol_M$.
\end{proposition}

\begin{lemma}\label{lmp}
Let $(M,\vartheta)$ be a local $G$-space, and suppose that
$\{g\}\times L\subseteq\operatorname{dom}(\vartheta)$ for some measurable
$L\subseteq M$ and $g\in G$. If $G$ is unimodular, then
\[
 \vol(g.L)=\vol(L).
\]
\end{lemma}

\begin{definition}\label{def:local-sofic-model}
Let $\mathcal M=(M,\vartheta)$ be a local $G$-space, let $U\subseteq G$ be
open and precompact, and let $\eps>0$. Define $M[U]=M[\vartheta,U]$ to be the
set of all $p\in M$ such that
\begin{enumerate}
\item if $g,h,gh\in U$, then $g.(h.p)=(gh).p$, with both sides defined;
\item the map $g\mapsto g.p$ is a homeomorphism from $U$ onto an open
neighborhood of $p$.
\end{enumerate}
If $0<\vol(M)<\infty$ and
\[
 \vol(M[U])\geq(1-\eps)\vol(M),
\]
then $\mathcal M$ is a \emph{$(U,\eps)$-sofic approximation} to $G$.
\end{definition}

\begin{definition}\label{def:sofic-approximation}
A \emph{sofic approximation} to $G$ is a sequence
$\Sigma=(M_i)_{i\geq1}$ such that each $M_i$ is a
$(U_i,\eps_i)$-sofic approximation, where
\begin{enumerate}
\item the $U_i$ are precompact open sets satisfying
$U_1\subseteq U_2\subseteq\cdots$ and $\bigcup_iU_i=G$;
\item $\eps_i\to0$ as $i\to\infty$.
\end{enumerate}
The group $G$ is \emph{sofic} if it admits such an approximation.
\end{definition}
Every locally compact sofic group is unimodular \cite{BowenBurton}.
Thus the fixed right Haar measure $\Haar$ is also left invariant
and invariant under inversion.
\subsection{The spatial microstate formula}

We use the following spatial microstate formula on finite-volume local
$G$-spaces. It has the same local-$G$-space form as Bowen's measure entropy
formula \cite[Definition~6.2]{BowenLCEntropy}, with a quasi-metric for
equivariance and weak-star neighborhoods for the empirical measure.

Bowen formulates his measure-conjugacy invariance theorem for uniform actions.
Since the compact metrizable continuous models used here need not satisfy this
additional uniformity assumption, we do not appeal directly to that theorem.
Instead, in \cref{app:singh-comparison} we compare the formula below with
Singh's compact-model pseudometric entropy. Singh uses a normalized $L^2$
metric and finitely many continuous test functions. We show that these
differences do not change the entropy. For nondiscrete groups, the resulting
identification with Singh's measure entropy gives the measure-conjugacy
invariance and compact-model independence needed in the proof of the main
theorem.

Let $G\curvearrowright(Y,\rho)$ be a jointly continuous action on a compact metric space, with $0\leq\rho\leq1$, and let $\nu$ be $G$-invariant. Fix a proper left-invariant compatible metric $d_G$ on $G$ and write
\[
 B(r):=\{g\in G:d_G(g,e)<r\}.
\]
If $M$ is a finite-volume local $G$-space, write
\[
 \overline{\vol}_M:=\frac{\vol_M}{\vol_M(M)},
 \qquad
 \rho^M(\psi,\psi'):=\int_M\rho(\psi(p),\psi'(p))\,d\overline{\vol}_M(p).
\]
For $g\in G$ let
\[
 M[g]:=\{p\in M:(g,p)\text{ lies in the domain of the partial action}\}.
\]
For a measurable map $\psi:M\to Y$ define
\begin{equation}\label{eq:defect}
 \Delta_{\rho,M}(g,\psi)
 :=\int_{M[g]}\rho\bigl(\psi(g.p),g\psi(p)\bigr)\,d\overline{\vol}_M(p)
   +\overline{\vol}_M(M\setminus M[g]).
\end{equation}
Undefined points contribute \(1\). In other words, this is Bowen's quasi-metric distance between the partially defined map $\psi\circ g$ and the everywhere defined map $g\circ\psi$.

Throughout the paper, weak-star neighborhoods are understood
to be open.
For a precompact open identity neighborhood $L\subset G$, $\delta>0$, and a weak-star neighborhood $\cO$ of $\nu$, define the standard model space
\[
 \Map(M,Y,\rho:L,\delta,\cO)
 :=\left\{\psi:M\to Y:
 \Delta_{\rho,M}(g,\psi)<\delta\ \forall g\in L,
 \quad \psi_*\overline{\vol}_M\in\cO\right\}.
\]
We also define the averaged model space
\[
 \Mapavg(M,Y,\rho:L,\delta,\cO)
 :=\left\{\psi:M\to Y:
 \begin{aligned}
 &\frac1{\Haar(L)}\int_L\Delta_{\rho,M}(g,\psi)\,d\Haar(g)<\delta,\\
 &\psi_*\overline{\vol}_M\in\cO
 \end{aligned}
 \right\}.
\]
When the empirical condition is omitted, we suppress $\cO$ from the notation.

\begin{definition}
	Let $(Y,\rho)$ be a pseudometric space and let $\eps>0$.
	A subset
	$F\subseteq Y$ is \emph{$(\rho,\eps)$-separated} if
	$\rho(x,y)>\eps$ for all distinct $x,y\in F$.
	For $E\subseteq Y$, define
	\[
	\Sep_\eps(E,\rho)
	:=
	\sup\bigl\{
	|F|:
	F\subseteq E \text{ is finite and }
	(\rho,\eps)\text{-separated}
	\bigr\}.
	\]
	We use the conventions
	\[
	\Sep_\eps(\varnothing,\rho)=0,\qquad
	\log 0=-\infty,\qquad
	\log(+\infty)=+\infty.
	\]
\end{definition}

\begin{definition}\label{def:spatial-entropy}
For $\Sigma=(M_i)$, put $V_i:=\vol_{M_i}(M_i)$.
The spatial measure sofic entropy of the compact model $(Y,\nu)$
relative to $\Sigma$ is defined by
\begin{align*}
 h^{\mathrm{meas}}_{\Sigma,\eps}(G,Y,\nu,\rho)
 &:={\inf}_{\cO\ni\nu}\,{\inf}_{L}\,{\inf}_{\delta>0}
 \limsup_{i\to\infty}\frac1{V_i}
 \log\Sep_\eps\bigl(\Map(M_i,Y,\rho:L,\delta,\cO),\rho^{M_i}\bigr),\\
 h^{\mathrm{meas}}_{\Sigma}(G,Y,\nu,\rho)
 &:=\lim_{\eps\downarrow0}
 h^{\mathrm{meas}}_{\Sigma,\eps}(G,Y,\nu,\rho),
\end{align*}
where $L$ ranges over precompact open identity neighborhoods of $G$.
Replacing $\Map$ by $\Mapavg$ defines
$h^{\mathrm{meas,avg}}_{\Sigma,\eps}$ and
$h^{\mathrm{meas,avg}}_\Sigma$.
\end{definition}
\begin{remark}
	The appendix compares the local-$G$-space formula above with Singh's
	compact-model pseudometric entropy for the completed approximation
	$\widetilde\Sigma$. When $G$ is nondiscrete, the results cited there identify
	this pseudometric entropy with Singh's measure entropy. Consequently, the
	value is invariant under measure conjugacy and is independent of the chosen
	compact metrizable continuous model.
\end{remark}

The next comparison connects the random construction with Bowen's pointwise spatial formula. It is the empirical-measure version of \cite[Proposition~4.10]{BowenLCEntropy}; we include the proof because preservation of the weak-star condition is essential below.

\begin{lemma}\label{lem:large-products}\cite[Lemma~4.11]{BowenLCEntropy}
Let $0<r_0<r_1<r_2$ and $0<\tau<1$. Assume $r_0+r_1<r_2$ and
\begin{equation}\label{eq:haar-balls-condition}
 \tau\Haar(B(r_2))<(1-\tau)\Haar(B(r_0)).
\end{equation}
If measurable sets $W_j\subset B(r_j)$ satisfy
\[
 \Haar(W_j)>(1-\tau)\Haar(B(r_j)),\qquad j=0,2,
\]
then
\[
 W_0W_2\supset B(r_1).
\]
\end{lemma}

\begin{proposition}\label{prop:avg-standard}
For every compact metrizable continuous $G$-system $(Y,\nu)$ and every compatible metric $\rho\leq1$,
\[
 h^{\mathrm{meas,avg}}_{\Sigma,\eps}(G,Y,\nu,\rho)
 =h^{\mathrm{meas}}_{\Sigma,\eps}(G,Y,\nu,\rho)
 \qquad(\eps>0).
\]
So we have
\[
 h^{\mathrm{meas,avg}}_\Sigma(G,Y,\nu,\rho)
 =h^{\mathrm{meas}}_\Sigma(G,Y,\nu,\rho).
\]
\end{proposition}

\begin{proof}
Since every pointwise microstate is averaged,
\[
 \Map(M,Y,\rho:L,\delta,\cO)
 \subseteq \Mapavg(M,Y,\rho:L,\delta,\cO),
\]
and so
\[
 h^{\mathrm{meas}}_{\Sigma,\eps}(G,Y,\nu,\rho)
 \leq h^{\mathrm{meas,avg}}_{\Sigma,\eps}(G,Y,\nu,\rho).
\]

For the reverse inequality, we show that one model space is eventually contained in the other. Fix a precompact open identity neighborhood $L\subset G$, a number $0<\delta<1$, and a weak-star neighborhood $\cO$ of $\nu$. We will choose a precompact open identity neighborhood $\widetilde L$, a tolerance $\widetilde\delta>0$, and an index $I$ such that
\begin{equation}\label{eq:avg-to-standard-inclusion}
 \Mapavg(M_i,Y,\rho:\widetilde L,\widetilde\delta,\cO)
 \subseteq
 \Map(M_i,Y,\rho:L,\delta,\cO)
 \qquad(i\geq I).
\end{equation}
We do not change the microstate, so the empirical-measure condition is preserved exactly.

By joint continuity of the action and compactness of $Y$, the map
\[
 G\times Y\times Y\longrightarrow\mathbb R,
 \qquad (h,x,y)\longmapsto\rho(hx,hy),
\]
is uniformly continuous on a sufficiently small compact neighborhood of $\{e\}\times Y\times Y$. So there are $r_0>0$ and $\kappa>0$ such that
\begin{equation}\label{eq:local-equicontinuity}
 \rho(x,y)<\kappa,\quad h\in B(r_0)
 \quad\Longrightarrow\quad
 \rho(hx,hy)<\delta/10.
\end{equation}
Choose $r_1>r_0$ with $L\subset B(r_1)$, and then choose $r_2>r_0+r_1$. Since $B(r_0)$ and $B(r_2)$ have finite positive Haar measure, we may choose $0<\tau<1$ so small that
\begin{equation}\label{eq:tau-choices}
 \tau\Haar(B(r_2))<(1-\tau)\Haar(B(r_0)),
 \qquad
 \tau<\min\{\kappa,\delta/10\}.
\end{equation}
Put
\[
 \widetilde L:=B(r_2),\qquad \widetilde\delta:=\tau^{10}.
\]
The balls are symmetric because
$d_G(g^{-1},e)=d_G(e,g)=d_G(g,e)$.

Because $U_i\uparrow G$, for all sufficiently large $i$ one has
$\widetilde L^2\subset U_i$. Monotonicity of the good sets then gives
$M_i[U_i]\subset M_i[\widetilde L^2]$, so the defining sofic
condition implies
$\dbar_{M_i}(M_i[\widetilde L^2])\to1$. Choose $i$ sufficiently large
that
\begin{equation}\label{eq:comparison-core}
 \dbar_{M_i}\bigl(M_i[\widetilde L^2]\bigr)
 >1-\min\{\tau^{10},\delta/10\}.
\end{equation}
Let
\[
 \psi\in\Mapavg(M_i,Y,\rho:\widetilde L,\tau^{10},\cO).
\]
For $u\in\widetilde L$, define
\[
 E(u):=\left\{p\in M_i[u]:
 \rho\bigl(u\psi(p),\psi(u.p)\bigr)<\tau^4\right\}
\]
and
\[
 W:=\left\{u\in\widetilde L:
 \dbar_{M_i}(E(u))>1-\tau^4\right\}.
\]
If $u\notin W$, then the complement of $E(u)$, including points at which $u.p$ is undefined, has normalized measure at least $\tau^4$. On the part of $M_i[u]\setminus E(u)$ the metric error is at least $\tau^4$, while on $M_i\setminus M_i[u]$ the contribution from the undefined part of the domain is $1\geq\tau^4$. This gives
\[
 \Delta_{\rho,M_i}(u,\psi)
 \geq \tau^4\bigl(1-\dbar_{M_i}(E(u))\bigr)
 \geq\tau^8.
\]
The averaged microstate inequality now gives
\begin{align*}
 \tau^{10}\Haar(\widetilde L)
 &>\int_{\widetilde L}\Delta_{\rho,M_i}(u,\psi)\,d\Haar(u)\\
 &\geq\tau^8\Haar(\widetilde L\setminus W).
\end{align*}
So we have
\begin{equation}\label{eq:good-group-elements}
 \Haar(\widetilde L\setminus W)<\tau^2\Haar(B(r_2)).
\end{equation}
In particular,
\[
 \Haar(W)>(1-\tau^2)\Haar(B(r_2))
 >(1-\tau)\Haar(B(r_2)).
\]
Also, using the first inequality in \eqref{eq:tau-choices},
\begin{align*}
 \Haar(B(r_0)\setminus W)
 &\leq \Haar(\widetilde L\setminus W)\\
 &<\tau^2\Haar(B(r_2))\\
 &<\tau(1-\tau)\Haar(B(r_0))
 <\tau\Haar(B(r_0)).
\end{align*}
So we have
\[
 \Haar(W\cap B(r_0))>(1-\tau)\Haar(B(r_0)).
\]
By \cref{lem:large-products}, every $g\in L\subset B(r_1)$ admits a factorization
\begin{equation}\label{eq:g-hk-factorization}
 g=hk,
 \qquad h\in W\cap B(r_0),\quad k\in W.
\end{equation}
Fix such a $g$ and such a factorization.

We now estimate the error of $\psi$ at $g$ directly. For a point $p\in M_i$, define three nonnegative functions $A_1,A_2,A_3$, each bounded by $1$, as follows. If the expressions involved are defined, let
\begin{align*}
 A_1(p)&:=\rho\bigl(h(k\psi(p)),h\psi(k.p)\bigr),\\
 A_2(p)&:=\rho\bigl(h\psi(k.p),\psi(h.(k.p))\bigr),\\
 A_3(p)&:=\rho\bigl(\psi(h.(k.p)),\psi((hk).p)\bigr);
\end{align*}
and set the corresponding $A_j(p)$ equal to $1$ whenever one of its required partial-action expressions is undefined. If $(hk).p$ is defined and all intermediate expressions are defined, the triangle inequality gives
\[
 \rho\bigl((hk)\psi(p),\psi((hk).p)\bigr)
 \leq A_1(p)+A_2(p)+A_3(p).
\]
If $(hk).p$ is undefined, then $A_3(p)=1$, so the contribution from the undefined part of the domain in \eqref{eq:defect} is also bounded by $A_1(p)+A_2(p)+A_3(p)$. If $(hk).p$ is defined but one of the intermediate expressions is undefined, then by definition at least one of $A_1(p),A_2(p),A_3(p)$ equals $1$, while the metric error is at most $1$. In every case, the same pointwise bound holds, so
\begin{equation}\label{eq:defect-three-direct}
 \Delta_{\rho,M_i}(g,\psi)
 \leq \int_{M_i}\bigl(A_1+A_2+A_3\bigr)\,d\dbar_{M_i}.
\end{equation}

Because $k\in W$, we have $\dbar_{M_i}(E(k))>1-\tau^4$. For $p\in E(k)$,
\[
 \rho\bigl(k\psi(p),\psi(k.p)\bigr)<\tau^4<\kappa.
\]
Since $h\in B(r_0)$, \eqref{eq:local-equicontinuity} gives $A_1(p)<\delta/10$. On the complement of $E(k)$ we only use $A_1\leq1$. This gives
\begin{equation}\label{eq:A1-bound}
 \int_{M_i}A_1\,d\dbar_{M_i}
 <\frac\delta{10}+\tau^4.
\end{equation}

To estimate $A_2$, we must use the fact that a partial translation preserves the canonical measure. Set
\[
 C_i:=M_i[\widetilde L^2],
 \qquad F:=E(h)\cap C_i.
\]
Since $h\in W$, \eqref{eq:comparison-core} gives
\begin{equation}\label{eq:F-lower-measure}
 \dbar_{M_i}(F)
 \geq1-\tau^4-\dbar_{M_i}(M_i\setminus C_i)
 >1-\tau^4-\delta/10.
\end{equation}
For every $q\in F$, the point $k^{-1}.q$ is defined because $q\in M_i[\widetilde L^2]$ and $k^{-1}\in\widetilde L$. Put
\[
 p:=k^{-1}.q.
\]
The local inverse and multiplication axioms on $M_i[\widetilde L^2]$ imply
\[
 k.p=k.(k^{-1}.q)=q.
\]
So we have
\[
 k^{-1}.F\subseteq\{p\in M_i:k.p\in E(h)\}.
\]
The partial translation $T(q):=k^{-1}.q$ is injective on $F$, with inverse $p\mapsto k.p$ on $T(F)$. Since $F\subset M_i[\widetilde L^2]$ and $k,k^{-1}\in\widetilde L$, \cref{lem:local-measure-preserving} applies and gives
\[
 \vol_{M_i}(k^{-1}.F)=\vol_{M_i}(F).
\]
Combining this identity with \eqref{eq:F-lower-measure}, we get
\begin{equation}\label{eq:pullback-Eh}
 \dbar_{M_i}\{p:k.p\in E(h)\}
 >1-\tau^4-\delta/10.
\end{equation}
Whenever $k.p\in E(h)$, all expressions in $A_2(p)$ are defined and
\[
 A_2(p)=\rho\bigl(h\psi(k.p),\psi(h.(k.p))\bigr)<\tau^4.
\]
Using $A_2\leq1$ elsewhere and \eqref{eq:pullback-Eh}, we get
\begin{equation}\label{eq:A2-bound}
 \int_{M_i}A_2\,d\dbar_{M_i}
 \leq \tau^4+\tau^4+\delta/10.
\end{equation}

Finally, if $p\in C_i=M_i[\widetilde L^2]$, then $h,k,hk\in\widetilde L^2$ and the local multiplication axiom gives
\[
 h.(k.p)=(hk).p.
\]
So $A_3(p)=0$ on $C_i$, and $A_3\leq1$ everywhere. By \eqref{eq:comparison-core},
\begin{equation}\label{eq:A3-bound}
 \int_{M_i}A_3\,d\dbar_{M_i}
 \leq\dbar_{M_i}(M_i\setminus C_i)<\delta/10.
\end{equation}

Substituting \eqref{eq:A1-bound}, \eqref{eq:A2-bound}, and \eqref{eq:A3-bound} into \eqref{eq:defect-three-direct} gives
\[
 \Delta_{\rho,M_i}(g,\psi)
 <\frac{3\delta}{10}+3\tau^4.
\]
Since $\tau<\delta/10<1$, we have $\tau^4<\delta/10$, and so
\[
 \Delta_{\rho,M_i}(g,\psi)<\frac{6\delta}{10}<\delta.
\]
The element $g\in L$ was arbitrary, so $\psi$ belongs to
$\Map(M_i,Y,\rho:L,\delta,\cO)$. This proves \eqref{eq:avg-to-standard-inclusion}.

For every $\eps>0$, the inclusion gives
\begin{align*}
 &\limsup_{i\to\infty}\frac1{V_i}\log\Sep_\eps
 \bigl(\Mapavg(M_i,Y,\rho:\widetilde L,\widetilde\delta,\cO),\rho^{M_i}\bigr)\\
 &\qquad\leq
 \limsup_{i\to\infty}\frac1{V_i}\log\Sep_\eps
 \bigl(\Map(M_i,Y,\rho:L,\delta,\cO),\rho^{M_i}\bigr).
\end{align*}
Taking the infimum over the averaged parameters on the left and then over $L$, $\delta$, and $\cO$ on the right proves
\[
 h^{\mathrm{meas,avg}}_{\Sigma,\eps}
 \leq h^{\mathrm{meas}}_{\Sigma,\eps}.
\]
Together with the first inequality, this proves equality at every scale $\eps$, and taking the limit as $\eps\downarrow0$ proves the equality of the full entropies.
\end{proof}


	\begin{definition}\label{def:cpe}
	A standard pmp action $G\curvearrowright(X,\mu)$ has
	\emph{completely positive measure sofic entropy relative to $\Sigma$} if
	every nontrivial measurable factor $G\curvearrowright(Z,\zeta)$ has
	strictly positive measure sofic entropy relative to $\Sigma$.
	When $G$ is nondiscrete, \cref{cor:singh-measure-identification} shows that
	this entropy may be computed from any compact metrizable continuous model
	$(Y,\nu)$ of the factor as
	\[
	h^{\mathrm{meas}}_\Sigma(G,Y,\nu,\rho),
	\]
	where $\rho\leq1$ is any compatible metric. The resulting value is
	independent of both the compact model and the metric.
\end{definition}
\begin{remark}\label{rem:cpe-model-independence}
This compact-model independence follows from
\cref{thm:identification-bowen} and Singh's
measure-conjugacy invariance; it does not follow just from
	topological-conjugacy invariance of the compact-model formula.
\end{remark}


\section{Local \texorpdfstring{$G$}{G}-space models}

Throughout the remainder of the paper, $G$ denotes the locally compact second countable sofic group in \cref{thm:main}. A local $G$-space, its canonical measure, and the good sets $M[U]$ are as in Section~2. The good sets $M[U]$ are Borel in the standard formulation. Fix a right Haar measure $\Haar$ on $G$.

Let $\Sigma=(M_i)$ be a sofic approximation. So there are precompact open identity neighborhoods $U_i\uparrow G$ and numbers $\eps_i\downarrow0$ such that
\[
 \vol_{M_i}(M_i[U_i])\geq (1-\eps_i)\vol_{M_i}(M_i).
\]
Write
\[
V_i:=\vol_{M_i}(M_i),\qquad \dbar_i:=V_i^{-1}\vol_{M_i}.
\]

\begin{lemma}\label{lem:local-measure-preserving}
Let $M$ be a local $G$-space with canonical measure.
For every $g\in G$, the partial translation
\[
T_g:M[g]\to M[g^{-1}],\qquad T_g(p)=g.p,
\]
is a Borel bijection with inverse $T_{g^{-1}}$, and for every Borel
$E\subset M[g]$ one has
\[
\operatorname{vol}_M(T_gE)=\operatorname{vol}_M(E).
\]
\end{lemma}

\begin{proof}
	The inverse axiom gives
	$T_{g^{-1}}T_g=\mathrm{id}$ on $M[g]$ and
	$T_gT_{g^{-1}}=\mathrm{id}$ on $M[g^{-1}]$.
	Since the partial action is continuous, both maps are continuous.
	Thus $T_g$ is a homeomorphism between these open subsets of $M$;
	in particular, it maps Borel sets to Borel sets.
	The measure-preserving identity follows from \cref{lmp},
	since $G$ is unimodular.
\end{proof}

\subsection{Divergence of model volume}

\begin{lemma}\label{lem:volume-divergence}
If $G$ is noncompact, then $V_i\to\infty$.
\end{lemma}

\begin{proof}
For all sufficiently large $i$, $\eps_i<1$, and so $M_i[U_i]$ is nonempty. Choose $p_i\in M_i[U_i]$. By the definition of $M_i[U_i]$, the orbit map
\[
 U_i\longrightarrow U_i.p_i,\qquad g\longmapsto g.p_i,
\]
is a homeomorphism onto an open neighborhood of $p_i$. By the defining property of the canonical measure,
\[
 \vol_{M_i}(U_i.p_i)=\Haar(U_i).
\]
So $V_i\geq \Haar(U_i)$. Since $U_i\uparrow G$, continuity from below gives
\[
 \lim_i\Haar(U_i)=\Haar(G).
\]
A locally compact group has finite Haar measure if and only if it is compact. So $\Haar(G)=\infty$, and the result follows.
\end{proof}

\subsection{Choice of local windows}

Choose symmetric precompact open neighborhoods
\[
 B_1\subset B_2\subset\cdots,\qquad \overline{B_n}\subset B_{n+1},\qquad \bigcup_nB_n=G,
\]
with $e\in B_1$. By passing to a slowly increasing integer sequence $r(i)\to\infty$, we may arrange that
\[
 W_i:=B_{r(i)}\quad\text{satisfies}\quad W_i^4\subset U_i.
\]
Define
\[
 A_i:=M_i[W_i^2],\qquad D_i:=M_i[W_i^4].
\]
Since $W_i^4\subset U_i$, the monotonicity of the good sets gives
\[
 M_i[U_i]\subset D_i\subset A_i.
\]
So we have
\begin{equation}\label{eq:deep-core-density}
 d_i:=\dbar_i(D_i)\geq1-\eps_i\longrightarrow1.
\end{equation}

The following Compatibility of local charts lemma explains the action convention used throughout the proof.

\begin{lemma}\label{lem:change-root}
Let $K,L\subset G$ be compact. For all sufficiently large $i$, the following holds. If $p\in D_i$, $g\in L$, and $q=g.p$, then $q\in A_i$ and
\[
 h.q=(hg).p\qquad\text{for every }h\in K.
\]
Also, the orbit maps $h\mapsto h.p$ and $h\mapsto h.q$ are injective on $W_i$.
\end{lemma}

\begin{proof}
Take $i$ so large that $K\cup L\cup KL\subset W_i$. Since $p\in M_i[W_i^4]$, the local multiplication identities at $p$ hold for every product used below. In particular,
\[
 h.(g.p)=(hg).p\qquad(\text{for every } h\in K).
\]
Put $q=g.p$. Next we check $q\in A_i=M_i[W_i^2]$.

For $a\in W_i^2$ we have $ag\in W_i^3\subset W_i^4$, and so
\[
 a.q=a.(g.p)=(ag).p.
\]
Right multiplication by $g$ is a homeomorphism from $W_i^2$ onto $W_i^2g$, while the orbit map $u\mapsto u.p$ is a homeomorphism on $W_i^4$. Their composition then shows that $a\mapsto a.q$ is a homeomorphism from $W_i^2$ onto an open neighborhood of $q$. If $a,b,ab\in W_i^2$, then the same calculation, now using $b g,abg\in W_i^4$, gives
\[
 a.(b.q)=a.(bg.p)=(abg).p=(ab).(g.p)=(ab).q.
\]
So $q\in M_i[W_i^2]=A_i$. The orbit maps at $p$ and $q$ are also injective on $W_i\subset W_i^2$.
\end{proof}

\section{Poisson configurations and the completion construction}

\begin{definition}[Poisson point process]
Let $X$ be a locally compact second countable Hausdorff space
with Borel $\sigma$-algebra $\mathcal B$, and let $\lambda$ be
a Radon measure on $X$.
	A Poisson point process with intensity measure $\lambda$ is a random
	locally finite counting measure $\eta$ on $X$ such that for every
	finite collection of pairwise disjoint measurable sets
	$A_1,\dots,A_k$ with $\lambda(A_j)<\infty$,
	$
	\eta(A_1),\dots,\eta(A_k)
	$
	are independent random variables and
	\[
	\eta(A_j)\sim \operatorname{Pois}(\lambda(A_j)).
	\]
	We write $\PPP(\lambda)$ for this law.
\end{definition}
Let $\cN(G)$ be the space of locally finite integer-valued Radon measures on $G$, with the vague topology and its Borel $\sigma$-algebra. An element can be written as
$
\omega = \sum_{x \in P} \delta_x,
$
where $P$ is a locally finite configuration, with multiplicities allowed. The vague topology is characterized by
\[
\omega_n \to \omega \quad \Leftrightarrow \quad \int f \, d\omega_n \to \int f \, d\omega, \quad \forall f \in C_c(G).
\]
Equipped with this topology, $\cN(G)$ is a Polish space. For $g\in G$ define
\begin{equation}\label{eq:right-shift}
 (g\omega)(E):=\omega(Eg),\qquad E\subset G\text{ Borel}.
\end{equation}
This is a continuous left action. Since $\Haar$ is right invariant, $\mu_\alpha$
is invariant under this action.

We fix a bounded compatible metric $\rho_X$ on $\cN(G)$ with the following locality property:

\begin{equation}\label{eq:local-metric}
 \text{for every }\eta>0\text{ there is compact }K\subset G\text{ such that }
 \omega|_K=\omega'|_K\Longrightarrow \rho_X(\omega,\omega')<\eta.
\end{equation}

	To define such a metric, choose a compact exhaustion
\[
K_1\subset\operatorname{int}K_2\subset\cdots,
\qquad \bigcup_{m\geq1}K_m=G.
\]
For every $m$, choose a cutoff $\chi_m\in C_c(G)$ satisfying
$0\leq\chi_m\leq1$, $\chi_m\equiv1$ on $K_m$, and
$\operatorname{supp}(\chi_m)\subseteq K_{m+1}$.  Also choose a countable
family dense, in the uniform norm, in the continuous functions supported in
$K_m$. Enumerate the union of all these dense families and all the cutoffs as
$(f_n)_{n\geq1}$, arranging that every finite initial segment has supports in
one compact set.

This family determines convergence. In fact, if
$\int f_n\,d\omega_j\to\int f_n\,d\omega$ for every $n$, then for
each $m$ the convergence of the cutoff coordinate gives
\[
\sup_j\omega_j(K_m)
\leq \sup_j\int\chi_m\,d\omega_j<\infty.
\]
The resulting uniform local-mass bound allows uniform approximation on
$K_m$ to pass from the chosen dense family to every $f\in C_c(G)$.
The formula
\[
\rho_X(\omega,\omega')
:=
\sum_{n\geq1}2^{-n}
\frac{
	\abs{\int f_n\dd\omega-\int f_n\dd\omega'}
}{
	1+\abs{\int f_n\dd\omega-\int f_n\dd\omega'}
}
\]
therefore defines a bounded compatible metric on $\cN(G)$. It has the required locality
property: given $\eta>0$, choose $N$ so large that
$\sum_{n>N}2^{-n}<\eta$ and let $K$ contain the supports of
$f_1,\dots,f_N$. If $\omega|_K=\omega'|_K$, then the first $N$ summands
vanish and $\rho_X(\omega,\omega')<\eta$.

For $p\in A_i$, let
\[
 \theta_{i,p}:W_i\longrightarrow W_i.p,\qquad h\longmapsto h.p.
\]
This is a homeomorphism. Its inverse gives the coordinate of a point in the chart based at \(p\).

\begin{lemma}\label{lem:borel-root}
For each $i$, the set
$
 \mathscr R_i:=\{(p,x)\in A_i\times M_i:x\in W_i.p\}
$
is Borel, and the map
\[
 c_i:\mathscr R_i\longrightarrow W_i,
 \qquad c_i(p,x)=\theta_{i,p}^{-1}(x),
\]
is Borel. It also follows that if $\cN(M_i)$ denotes the standard Borel space of locally finite integer-valued Radon measures on $M_i$, then
\[
 (p,\xi)\longmapsto (c_i(p,\cdot))_*(\xi|_{W_i.p})
\]
is a Borel map from $A_i\times\cN(M_i)$ to $\cN(W_i)$.
\end{lemma}

\begin{proof}
The set $A_i=M_i[W_i^2]$ is Borel by the standard measurability of good sets in a local $G$-space. Consider
\[
 T_i:W_i\times A_i\longrightarrow A_i\times M_i,
 \qquad T_i(h,p)=(p,h.p).
\]
The partial action is continuous on its domain, so $T_i$ is continuous. It is injective because, for every $p\in A_i=M_i[W_i^2]$, the orbit map $h\mapsto h.p$ is injective on $W_i$. Both $W_i\times A_i$ and $A_i\times M_i$ are standard Borel spaces. The Lusin--Souslin theorem then implies that the image
\[
 \mathscr R_i=T_i(W_i\times A_i)
 =\{(p,x):p\in A_i,\ x\in W_i.p\}
\]
is Borel and that $T_i^{-1}:\mathscr R_i\to W_i\times A_i$ is Borel. The first coordinate of $T_i^{-1}(p,x)$ is exactly $c_i(p,x)$, so $c_i$ is Borel.

For the claim concerning point measures, we use the following elementary fact about integration against a varying point measure. If $P$ and $X$ are standard Borel spaces and $u:P\times X\to[0,\infty]$ is Borel, then
\begin{equation}\label{eq:point-measure-param-int}
 (p,\xi)\longmapsto\int_Xu(p,x)\,d\xi(x)
\end{equation}
is Borel on $P\times\cN(X)$. To verify this, first take $u=\1_{A\times B}$. Then \eqref{eq:point-measure-param-int} equals $\1_A(p)\xi(B)$, which is Borel because the evaluation maps $\xi\mapsto\xi(B)$ generate the standard Borel structure on $\cN(X)$. The class of nonnegative Borel functions for which the claim holds is closed under nonnegative linear combinations and monotone limits. The monotone class theorem then proves the claim for every nonnegative Borel $u$. Bounded signed functions follow by applying the result to the positive and negative parts.

Fix $f\in C_c(W_i)$. Extend the function
\[
 (p,x)\longmapsto \1_{\mathscr R_i}(p,x)f(c_i(p,x))
\]
by zero outside $\mathscr R_i$. It is a bounded Borel function on $A_i\times M_i$. For every $(p,\xi)$,
\begin{align*}
 &\int_{W_i}f(h)\,d\bigl[(c_i(p,\cdot))_*(\xi|_{W_i.p})\bigr](h)\\
 &\qquad=\int_{M_i}\1_{\mathscr R_i}(p,x)f(c_i(p,x))\,d\xi(x).
\end{align*}
The right-hand side is Borel in $(p,\xi)$ by the previous integration fact. It is finite because $\operatorname{supp}f$ is compact and its image $(\operatorname{supp}f).p$ is compact, while $\xi$ is locally finite.

Finally, choose a countable convergence-determining family $(f_m)$ in $C_c(W_i)$. The standard Borel structure of $\cN(W_i)$ is generated by the maps
\[
 \eta\longmapsto\int f_m\,d\eta.
\]
Each of these coordinate functions is Borel after composition with
\[
(p,\xi)\mapsto (c_i(p,\cdot))_*(\xi|_{W_i.p}).
\]
Therefore this map is Borel.
\end{proof}

Let
\[
 \Xi_i\sim\PPP(\alpha\vol_{M_i})
\]
be a Poisson point process on $M_i$, and independently let
\[
 Z_i\sim\PPP(\alpha\Haar)
\]
be a Poisson point process on $G$. Define the random map
$
 \Phi_i:M_i\longrightarrow\cN(G)
$
by
\begin{equation}\label{eq:completed-root}
 \Phi_i(p):=
 \begin{cases}
 (\theta_{i,p}^{-1})_*(\Xi_i|_{W_i.p})+Z_i|_{G\setminus W_i},&p\in A_i,\\[2mm]
 Z_i,&p\notin A_i.
 \end{cases}
\end{equation}
This construction takes values in $\cN(G)$.
Indeed, since $W_i$ is a symmetric open identity neighborhood,
$
\overline{W_i}\subseteq W_i^2.
$
For $p\in A_i$, the orbit map is continuous on $W_i^2$, so
$(\overline{W_i}).p$ is compact.
Thus every $\xi\in\cN(M_i)$ satisfies
$
\xi(W_i.p)<\infty,
$
and the pulled-back interior measure is finite.
Its sum with $z|_{G\setminus W_i}$ is therefore locally finite
for every $z\in\cN(G)$.

The construction is also jointly Borel.
For $p\in A_i$ and $f\in C_c(G)$, integration against the
completed measure gives
\[
\int_{M_i}
\1_{\mathscr R_i}(p,x)f(c_i(p,x))\,d\xi(x)
+
\int_{G\setminus W_i}f(h)\,dz(h),
\]
where the first integrand is extended by zero outside
$\mathscr R_i$.
Both terms are Borel functions of $(p,\xi,z)$ by
\cref{lem:borel-root} and the parameterized integration
argument in its proof.
A countable family of functions in $C_c(G)$ determines the
Borel structure of $\cN(G)$.
Since $A_i$ is Borel, the two cases in
\eqref{eq:completed-root} define a jointly Borel map
$(\omega,p)\mapsto\Phi_i(\omega,p)$.

\begin{proposition}\label{prop:exact-marginal}
For every $i$ and every $p\in M_i$,
\[
 \Phi_i(p)\sim\mu_\alpha.
\]
In particular, for every bounded Borel $F:\cN(G)\to\mathbb R$,
\begin{equation}\label{eq:exact-expectation}
 \mathbb E[F(\Phi_i(p))]=\int F\dd\mu_\alpha.
\end{equation}
\end{proposition}

\begin{proof}
If $p\notin A_i$, the claim is clear. Suppose $p\in A_i$. The restriction $\Xi_i|_{W_i.p}$ is a Poisson process with intensity $\alpha\vol_{M_i}|_{W_i.p}$. By the defining property of the canonical measure and the mapping theorem for Poisson processes, its pullback through $\theta_{i,p}$ is a Poisson process on $W_i$ with intensity $\alpha\Haar|_{W_i}$. It is independent of $Z_i|_{G\setminus W_i}$, which is Poisson with intensity $\alpha\Haar|_{G\setminus W_i}$. The union of the two independent configurations on the disjoint sets
\[
G=W_i\sqcup(G\setminus W_i)
\]
therefore has law \(\PPP(\alpha\Haar)\).

The term \(Z_i|_{G\setminus W_i}\) adds the Poisson points outside \(W_i\). Without it, \(\Phi_i(p)\) would have no points there, which occurs only on a set of \(\mu_\alpha\)-measure zero. A measurable factor map is only determined up to sets of measure zero, so its values on such a set may be arbitrary. In contrast, the completed configuration has law \(\mu_\alpha\), and hence lies almost surely in the set where the factor map has the required properties.
\end{proof}

\begin{proposition}\label{prop:local-covariance}
Let $K,L\subset G$ be compact. For all sufficiently large $i$, for every $p\in D_i$ and $g\in L$,
\begin{equation}\label{eq:local-covariance}
 \Phi_i(g.p)|_K=(g\Phi_i(p))|_K
\end{equation}
almost surely.
\end{proposition}

\begin{proof}
For large $i$, \cref{lem:change-root} applies, $K\cup KL\subset W_i$, and $q=g.p$ belongs to $A_i$. The exterior process $Z_i|_{G\setminus W_i}$ contributes no points to $K$ and no points to $Kg$. So both sides of \eqref{eq:local-covariance} are determined by $\Xi_i$.

For a Borel set $E\subset K$, the left-hand side counts $x\in\Xi_i$ for which $x=h.q$ with $h\in E$. By \cref{lem:change-root}, $h.q=(hg).p$. So this equals the number of points whose coordinates in the chart based at \(p\) lie in \(Eg\), that is
\[
 \Phi_i(q)(E)=\Phi_i(p)(Eg)=(g\Phi_i(p))(E)
\]
by \eqref{eq:right-shift}. So the restrictions agree, as claimed.
\end{proof}

\section{Exponential concentration for local observables}

Let $K \subseteq G$ be a compact set. A bounded Borel function
\[
F : \cN(G) \to \mathbb{R}
\]
is called $K$-local if $F(\omega)$ depends only on $\omega|_K$, for any $\omega\in \cN(G)$. That is, if
$\omega|_K = \omega'|_K$,
then
$F(\omega) = F(\omega')$.

We begin with the geometric overlap estimate.

\begin{lemma}\label{lem:overlap}
Let $K \subseteq G$ be a compact set. For all sufficiently large $i$ and every $z\in M_i$,
\begin{equation}\label{eq:overlap}
 \vol_{M_i}\bigl(\{p\in D_i:z\in K.p\}\bigr)\leq \Haar(K).
\end{equation}
\end{lemma}

\begin{proof}
Let $E_z:=\{p\in D_i:z\in K.p\}$. If $E_z=\varnothing$ there is nothing to prove. Otherwise choose $p_0\in E_z$ and $k_0\in K$ with $z=k_0.p_0$. By \cref{lem:change-root}, applied with the fixed compact set $K$, we have $z\in A_i=M_i[W_i^2]$ for all sufficiently large $i$.

For $p\in E_z$, injectivity of the orbit map at $p$ gives a unique $k(p)\in K$ such that $z=k(p).p$. The inverse axiom for the partial action gives
\[
 p=k(p)^{-1}.z.
\]
So we have
\[
 E_z\subset K^{-1}.z.
\]
Since $z\in M_i[W_i^2]$ and $K^{-1}\subset W_i^2$ for large $i$, the single orbit chart centered at $z$ is defined and injective on all of $K^{-1}$. The canonical-measure formula in this chart gives
\[
 \vol_{M_i}(E_z)\leq \vol_{M_i}(K^{-1}.z)
 =\Haar(K^{-1})=\Haar(K),
\]
where the last equality follows from unimodularity.
\end{proof}

We use the following consequence of Wu's modified logarithmic Sobolev inequality \cite{Wu}. Since the precise Bernstein form is needed later, we include the complete Herbst argument; the logarithmic Sobolev inequality used in the first step is stated in \cite[Theorem~1.1]{BachmannPeccati}.

\begin{lemma}
	\label{lem:poisson-bd}
	Let $\lambda$ be a finite Radon measure on a locally compact
	second countable Hausdorff space $S$, and let
	$\eta\sim\PPP(\lambda)$.
	Let $H=H(\eta)$ be a bounded measurable functional.
	Suppose that, for some $c\geq0$,
	\begin{equation}\label{eq:add-one-bound}
		\abs{H(\xi+\delta_z)-H(\xi)}\leq c
	\end{equation}
	for every finite counting measure $\xi$ on $S$ and every $z\in S$.
	
	If $c=0$, then $H(\eta)$ is almost surely constant.
	If $c>0$, then, for every $t>0$,
	\begin{equation}\label{eq:poisson-bd}
		\mathbb P\bigl(\abs{H-\mathbb EH}\geq t\bigr)
		\leq
		2\exp\left(
		-\frac{t^2}
		{2\lambda(S)c^2+2ct/3}
		\right).
	\end{equation}
\end{lemma}

\begin{proof}
	For a finite counting measure $\xi$ and $z\in S$, write
	$
	D_zH(\xi):=H(\xi+\delta_z)-H(\xi).
	$
	
	We first consider the case $c=0$. Then
	$
	H(\xi+\delta_z)=H(\xi)
	$
	for every finite counting measure $\xi$ and every $z\in S$.
	Since $\lambda(S)<\infty$, the Poisson process $\eta$ has finitely
	many points almost surely. Removing its points one at a time therefore
	gives
$
	H(\eta)=H(0)
$ almost surely.
	Thus $H(\eta)$ is almost surely constant.
	
	From now on, assume that $c>0$.
	If $\lambda(S)=0$, then $\eta=0$ almost surely, and again $H(\eta)$
	is almost surely constant. Thus we may also assume that
	$
	\lambda(S)>0.
	$
	
	We first prove the upper-tail estimate.
	For $\theta\geq0$, put
	$
	L(\theta):=\log\mathbb E e^{\theta H}.
	$
	Since $H$ is bounded, $L$ is differentiable and all exponential
	moments below are finite.
	
	Wu's modified logarithmic Sobolev inequality, in the form of
	\cite[Theorem~1.1]{BachmannPeccati}, gives
	\begin{equation}\label{eq:wu-mlsi-applied}
		\operatorname{Ent}(e^{\theta H})
		\leq
		\mathbb E\left[
		e^{\theta H}
		\int_S
		\psi\bigl(\theta D_zH\bigr)\,d\lambda(z)
		\right],
	\end{equation}
	where
	$
	\operatorname{Ent}(X)
	:=
	\mathbb E[X\log X]
	-
	\mathbb E[X]\log\mathbb E[X],
	$
	and
	$
	\psi(u):=ue^u-e^u+1.
	$
	
	The function $\psi$ is nonnegative. Moreover, for every $a\geq0$,
	\begin{equation}\label{eq:psi-max}
		\max_{|u|\leq a}\psi(u)=\psi(a).
	\end{equation}
	Indeed,
	$
	\psi'(u)=ue^u,
	$
	so $\psi$ decreases on $(-\infty,0]$ and increases on $[0,\infty)$.
	It remains only to compare the two endpoints. We have
	$
	\psi(a)-\psi(-a)
	=
	2(a\cosh a-\sinh a).
	$
	Since
\[
	\frac{d}{da}(a\cosh a-\sinh a)
	=
	a\sinh a\geq0
\]
	and the expression vanishes at $a=0$, it follows that
$
	\psi(a)\geq\psi(-a).
$
	This proves \eqref{eq:psi-max}.
	
	By \eqref{eq:add-one-bound},
	$
	|D_zH|\leq c.
	$
	Hence \eqref{eq:wu-mlsi-applied} and \eqref{eq:psi-max} imply
	\begin{equation}\label{eq:entropy-bound-poisson}
		\operatorname{Ent}(e^{\theta H})
		\leq
		\lambda(S)\psi(\theta c)\,
		\mathbb E e^{\theta H}.
	\end{equation}
	
	On the other hand,
	\[
	\frac{\operatorname{Ent}(e^{\theta H})}
	{\mathbb E e^{\theta H}}
	=
	\theta L'(\theta)-L(\theta).
	\]
	Therefore, for $\theta>0$,
	\begin{equation}\label{eq:Herbst-differential}
		\frac{d}{d\theta}
		\left(\frac{L(\theta)}{\theta}\right)
		=
		\frac{\theta L'(\theta)-L(\theta)}{\theta^2}
		\leq
		\lambda(S)\frac{\psi(\theta c)}{\theta^2}.
	\end{equation}
	
	Since
	\[
	\frac{L(\theta)}{\theta}
	\longrightarrow
	\mathbb EH
	\qquad\text{as }\theta\downarrow0,
	\]
	integrating \eqref{eq:Herbst-differential} from $0$ to $\theta$
	gives
	\begin{equation}\label{eq:poisson-mgf-bound}
		\log\mathbb E
		e^{\theta(H-\mathbb EH)}
		\leq
		\lambda(S)
		\bigl(e^{\theta c}-1-\theta c\bigr).
	\end{equation}
	Here we used
	\[
	\frac{d}{d\theta}
	\left(
	\frac{e^{\theta c}-1-\theta c}{\theta}
	\right)
	=
	\frac{\psi(\theta c)}{\theta^2}.
	\]
	
	Let $u>0$. By Chernoff's inequality and
	\eqref{eq:poisson-mgf-bound},
	\begin{align*}
		\mathbb P(H-\mathbb EH\geq u)
		&\leq
		\inf_{\theta>0}
		\exp\left(
		-\theta u
		+
		\lambda(S)
		\bigl(e^{\theta c}-1-\theta c\bigr)
		\right).
	\end{align*}
	Set
	$
	v:=\frac{u}{\lambda(S)c}.
	$
	The minimizing value is
	$
	\theta=\frac1c\log(1+v).
	$
	Substitution yields the Bennett bound
	\begin{equation}\label{eq:Bennett-upper}
		\mathbb P(H-\mathbb EH\geq u)
		\leq
		\exp\bigl(-\lambda(S)h(v)\bigr),
	\end{equation}
	where
	\[
	h(v):=(1+v)\log(1+v)-v.
	\]
	
	Using the elementary inequality
	\begin{equation}\label{eq:h-Bernstein}
		h(v)
		\geq
		\frac{v^2}{2(1+v/3)},
		\qquad v\geq0,
	\end{equation}
	we obtain
	\begin{equation}\label{eq:poisson-upper-Bernstein}
		\mathbb P(H-\mathbb EH\geq u)
		\leq
		\exp\left(
		-\frac{u^2}
		{2\lambda(S)c^2+2cu/3}
		\right).
	\end{equation}
	
	Apply the same argument to $-H$.
	Since
	$
	|D_z(-H)|=|D_zH|\leq c,
	$
	we also have
	\[
	\mathbb P(\mathbb EH-H\geq u)
	\leq
	\exp\left(
	-\frac{u^2}
	{2\lambda(S)c^2+2cu/3}
	\right).
	\]
	Adding the two one-sided estimates proves
	\eqref{eq:poisson-bd}.
\end{proof}

\begin{proposition}\label{prop:local-concentration}
Let $K\subset G$ be compact and let $F:\cN(G)\to[-1,1]$ be $K$-local. Define
\[
 \mathcal A_i(F):=\frac1{V_i}\int_{D_i}F(\Phi_i(p))\dd\vol_{M_i}(p).
\]
Then
\begin{equation}\label{eq:local-mean}
 \mathbb E\mathcal A_i(F)=d_i\int F\dd\mu_\alpha.
\end{equation}
Also, for every $t>0$ there is $c=c(t,K,\alpha)>0$ such that, for all sufficiently large $i$,
\begin{equation}\label{eq:local-concentration}
 \mathbb P\left(\abs{\mathcal A_i(F)-\mathbb E\mathcal A_i(F)}>t\right)
 \leq2e^{-cV_i}.
\end{equation}
The same claim holds for a function taking values in $[-1,1]$ that is
$K$-local in each coordinate of two independent completed Poisson configuration models.
\end{proposition}

\begin{proof}
The expectation formula follows from \cref{prop:exact-marginal} and Fubini's theorem.
\begin{align*}
	\mathbb E\mathcal A_i(F) &= \frac{1}{V_i} \int_{D_i} \mathbb E[F(\Phi_i(p))] \dd\vol_{M_i}(p) \\
	&= \frac{1}{V_i} \int_{D_i} \left( \int F \dd\mu_\alpha \right) \mathrm{d}\vol_{M_i}(p) \\
	&= \frac{\vol_{M_i}(D_i)}{V_i} \int F \dd\mu_\alpha \\
	&= d_i \int F \dd\mu_\alpha.
\end{align*}
If $\Haar(K)=0$, then a Poisson process with law $\mu_\alpha$
has no points in $K$ almost surely.
Since $F$ is $K$-local, \cref{prop:exact-marginal} and
Fubini's theorem give
$\mathcal A_i(F)=d_iF(0)$ almost surely, where $0$ denotes
the empty configuration.
For an observable that is $K$-local in both coordinates,
the same argument gives the constant $d_iF(0,0)$.
Thus \eqref{eq:local-concentration} holds with any $c>0$
in this case.
We may therefore assume that $\Haar(K)>0$.

For large $i$, $K\subset W_i$. On $D_i$, the value $F(\Phi_i(p))$ depends only on the restriction of $\Xi_i$ to $K.p$ and is independent of $Z_i$.

Add a point at $z\in M_i$ to $\Xi_i$. The integrand can change only for $p\in D_i$ satisfying $z\in K.p$, and it changes by at most $2$. By \cref{lem:overlap},
\[
 \abs{D_z\mathcal A_i(F)}\leq\frac{2\vol_{M_i}\bigl(\{p\in D_i:z\in K.p\}\bigr)}{V_i}\leq\frac{2\Haar(K)}{V_i}=:c_i.
\]
The intensity measure of $\Xi_i$ has total mass $\alpha V_i$. Applying \cref{lem:poisson-bd} with $c=c_i$ and $\lambda(S)=\alpha V_i$ gives
\[
 \mathbb P\left(\abs{\mathcal A_i(F)-\mathbb E\mathcal A_i(F)}>t\right)
 \leq2\exp\left(
 -\frac{t^2V_i}{8\alpha\Haar(K)^2+4t\Haar(K)/3}
 \right).
\]
This gives the stated exponential bound. For a bounded function that is $K$-local in each coordinate of two independent models, view the pair $(\Xi_i,\Xi_i')$ as one Poisson process on the disjoint union $M_i\sqcup M_i$. Its intensity has total mass $2\alpha V_i$. Adding one point to either configuration affects the normalized average only at those \(p\in D_i\) whose window \(K.p\) contains that point. Hence the same overlap argument shows that the normalized average changes by at most \(2\Haar(K)/V_i\).  Applying \cref{lem:poisson-bd} gives explicitly
\[
 2\exp\left(
 -\frac{t^2V_i}{16\alpha\Haar(K)^2+4t\Haar(K)/3}
 \right),
\]
which is $2e^{-cV_i}$ for a constant $c=c(t,K,\alpha)>0$.
\end{proof}

\begin{remark}[Uniformity of the concentration bound]
The exponential rate in \cref{prop:local-concentration} is uniform over all
$K$-local observables $F$ with $\|F\|_\infty\leq1$. In particular, it depends
only on the window $K$, the intensity $\alpha$, and the deviation parameter
$t$, and not on the particular observable. The analogous two-copy estimate
has the same uniformity.
\end{remark}

\section{Random microstates for arbitrary measurable factors}

Let
$
 \pi:(\cN(G),\mu_\alpha)\longrightarrow(Y,\nu)
$
be a measurable factor, where $(Y,\nu)$ is a compact metrizable continuous model as given by \cref{prop:compact-model}. Fix a compatible metric $\rho_Y\leq1$.

Choose a Borel representative of the measurable factor map $\pi$. For every fixed $g\in G$, the factor property gives
$
 \pi(g\omega)=g\pi(\omega)
$
for $\mu_\alpha$-almost every $\omega$. The failure set
\[
 N:=\{(g,\omega)\in G\times\cN(G):
 \pi(g\omega)\neq g\pi(\omega)\}
\]
is Borel because both actions and $\pi$ are Borel. Every vertical section $N_g$ has $\mu_\alpha$-measure zero. Since Haar measure is $\sigma$-finite, Tonelli's theorem gives
\[
 (\Haar\times\mu_\alpha)(N)=0.
\]
So our Borel representative satisfies
\begin{equation}\label{eq:factor-ae}
 \pi(g\omega)=g\pi(\omega)
\end{equation}
for $(\Haar\times\mu_\alpha)$-almost every $(g,\omega)$. Below we use the factor identity for each fixed $g$, together
with the exact marginal law of the random configurations. Define
\[
 \Psi_i:=\pi\circ\Phi_i:M_i\longrightarrow Y.
\]
\begin{proposition}
\label{prop:factor-marginal}
For every $i$ and every $p\in M_i$,
\[
 \Psi_i(p)\sim\nu.
\]
In other words, for every bounded Borel function $f:Y\to\mathbb R$,
\[
 \mathbb E[f(\Psi_i(p))]=\int_Y f\,d\nu.
\]
\end{proposition}

\begin{proof}
By \cref{prop:exact-marginal}, $\Phi_i(p)$ has law $\mu_\alpha$. This gives
\[
 (\pi\circ\Phi_i(p))_*\mathbb P=\pi_*\mu_\alpha=\nu.
\]
\end{proof}

\subsection{Empirical convergence}

\begin{lemma}\label{lem:local-L1}
For every bounded Borel function $H:\cN(G)\to\mathbb R$ and every $\eta>0$, there are a compact $K\subset G$ and a bounded $K$-local Borel function $H_K$ such that
\[
 \norm{H-H_K}_{L^1(\mu_\alpha)}<\eta,
 \qquad \norm{H_K}_\infty\leq\norm{H}_\infty.
\]
\end{lemma}

\begin{proof}
Choose an increasing compact exhaustion
\[
 K_1\subseteq K_2\subseteq\cdots,
 \qquad \bigcup_{n\geq1}K_n=G,
\]
and let $\mathscr F_n$ be the $\sigma$-algebra generated by the restriction map
\[
 r_{K_n}:\omega\longmapsto\omega|_{K_n}.
\]
Compact restrictions generate the Borel $\sigma$-algebra of the configuration
space, so
\[
 \bigvee_{n\geq1}\mathscr F_n=\cB(\cN(G)).
\]
Let $H_n:=\mathbb E(H\mid\mathscr F_n)$.
The martingale convergence theorem gives $H_n\to H$ in $L^1(\mu_\alpha)$. Conditional expectation is a contraction on $L^\infty$, so
\[
 \|H_n\|_\infty\leq\|H\|_\infty.
\]
Choose $n$ such that $\|H-H_n\|_1<\eta$. Since $\mathscr F_n$ is generated by the Borel restriction map
\[
 r_{K_n}:\cN(G)\to\cN(K_n),\qquad\omega\mapsto\omega|_{K_n},
\]
the Doob--Dynkin lemma supplies a Borel function $\widetilde H_n$ on $\cN(K_n)$ such that
\[
 H_n(\omega)=\widetilde H_n(\omega|_{K_n})
\]
for $\mu_\alpha$-almost every $\omega$. After truncating $\widetilde H_n$ to the interval $[-\|H\|_\infty,\|H\|_\infty]$, this equality and the norm bound remain valid. Define
\[
 H_K(\omega):=\widetilde H_n(\omega|_{K_n}),\qquad K:=K_n.
\]
Then $H_K$ is Borel, $K$-local, satisfies $\|H_K\|_\infty\leq\|H\|_\infty$, and obeys $\|H-H_K\|_1<\eta$.
\end{proof}

\begin{proposition}\label{prop:empirical-factor}
For every weak-star neighborhood $\cO$ of $\nu$, as $i \to \infty$,
\[
 \mathbb P\bigl((\Psi_i)_*\dbar_i\in\cO\bigr)\longrightarrow1.
\]
\end{proposition}

\begin{proof}
It is enough to treat a basic weak-star neighborhood
\[
 \cO=\left\{\lambda\in\Prob(Y):
 \left|\int f_j\,d\lambda-\int f_j\,d\nu\right|<\tau
 \text{ for }1\leq j\leq m\right\},
\]
where $f_j\in C(Y)$ and $\|f_j\|_\infty\leq1$. Put $H_j:=f_j\circ\pi$. To prove convergence in probability, fix an
arbitrary $\xi>0$. Choose $\eta>0$ so small that
\begin{equation}\label{eq:empirical-eta-choice}
 \sqrt\eta+2\eta<\tau/3,
 \qquad m\sqrt\eta<\xi/2.
\end{equation}
By \cref{lem:local-L1}, for each $j$ there is a compact window and a local approximation with $L^1$ error less than $\eta$. Replacing the finitely many windows by their union, we get one compact $K\subset G$ and $K$-local Borel functions $H_{j,K}$ satisfying
\[
 \|H_j-H_{j,K}\|_{L^1(\mu_\alpha)}<\eta,
 \qquad \|H_{j,K}\|_\infty\leq1
 \quad(1\leq j\leq m).
\]

Define
\[
 E_{i,j}:=\frac1{V_i}\int_{M_i}
 |H_j(\Phi_i(p))-H_{j,K}(\Phi_i(p))|\,d\vol_{M_i}(p).
\]
The exact marginal identity and Tonelli's theorem give
\[
 \mathbb E E_{i,j}
 =\int_{\cN(G)}|H_j-H_{j,K}|\,d\mu_\alpha<\eta.
\]
So we have
\begin{equation}\label{eq:empirical-approx-prob}
 \mathbb P(E_{i,j}>\sqrt\eta)<\sqrt\eta
\end{equation}
by Markov's inequality.

Set
\[
 J_{i,j}:=\frac1{V_i}\int_{D_i}H_{j,K}(\Phi_i(p))\,d\vol_{M_i}(p),
 \qquad m_{j,K}:=\int H_{j,K}\,d\mu_\alpha.
\]
By \cref{prop:local-concentration},
\[
 J_{i,j}-d_i m_{j,K}\xrightarrow{\mathbb P}0.
\]
Since $d_i\to1$, we get $J_{i,j}\to m_{j,K}$ in probability.
So there is $i_1$ such that, for every $i\geq i_1$,
\begin{equation}\label{eq:empirical-local-prob}
 \sum_{j=1}^m
 \mathbb P\left(|J_{i,j}-m_{j,K}|>\tau/3\right)<\xi/2.
\end{equation}
The contribution from the complement of the set $D_i$ satisfies
\[
 \left|\frac1{V_i}\int_{M_i\setminus D_i}
 H_{j,K}(\Phi_i(p))\,d\vol_{M_i}(p)\right|
 \leq1-d_i.
\]
Also,
\[
 |m_{j,K}-\int H_j\,d\mu_\alpha|<\eta.
\]
On the event $E_{i,j}\leq\sqrt\eta$,
\begin{align*}
 &\left|\frac1{V_i}\int_{M_i}H_j(\Phi_i(p))\,d\vol_{M_i}(p)
 -\int H_j\,d\mu_\alpha\right|\\
 &\quad\leq E_{i,j}+|J_{i,j}-m_{j,K}|+(1-d_i)+\eta.
\end{align*}
Increase $i_1$ if necessary so that
$1-d_i<\tau/3-\sqrt\eta-\eta$ for every $i\geq i_1$, which is possible by
\eqref{eq:empirical-eta-choice}. For such $i$, failure of any one of the $m$
required moment inequalities is contained in the union of the events
\[
 \{E_{i,j}>\sqrt\eta\}
 \quad\text{and}\quad
 \{|J_{i,j}-m_{j,K}|>\tau/3\},
 \qquad 1\leq j\leq m.
\]
By \eqref{eq:empirical-approx-prob},
\eqref{eq:empirical-local-prob}, and the union bound,
\[
 \mathbb P\bigl((\Psi_i)_*\dbar_i\notin\cO\bigr)
 \leq m\sqrt\eta+\xi/2<\xi.
\]
Here we used $\int H_j\,d\mu_\alpha=\int f_j\,d\nu$. Since $\xi>0$ was
arbitrary, the required probability tends to one.
\end{proof}

\subsection{Approximate equivariance}

Recall from \eqref{eq:defect} that a map $\psi:M_i\to Y$ is $(L,\delta)$-equivariant on average if
\[
 \frac1{\Haar(L)}\int_L\Delta_{\rho_Y,M_i}(g,\psi)\dd\Haar(g)<\delta.
\]

\begin{proposition}\label{prop:avg-equivariance}
For every precompact identity neighborhood $L\subset G$, as $i \to \infty$,
\[
 \frac1{\Haar(L)}\int_L
 \Delta_{\rho_Y,M_i}(g,\Psi_i)\dd\Haar(g)
 \xrightarrow{\mathbb P}0.
\]
\end{proposition}

\begin{proof}
Fix $\eta>0$. By Lusin's theorem, there is a compact set $C\subset\cN(G)$ such that
\[
 \mu_\alpha(C)>1-\eta
\]
and $\pi|_C$ is continuous. Since $C$ is compact, the restriction $\pi|_C$ is uniformly continuous. So there is $\delta_X>0$ such that
\begin{equation}\label{eq:lusin-modulus}
 x,x'\in C,\quad \rho_X(x,x')<\delta_X
 \quad\Longrightarrow\quad
 \rho_Y(\pi(x),\pi(x'))<\eta.
\end{equation}
By the locality property \eqref{eq:local-metric}, choose a compact set $K\subset G$ such that
\begin{equation}\label{eq:agreement-K-deltaX}
 x|_K=x'|_K\quad\Longrightarrow\quad\rho_X(x,x')<\delta_X.
\end{equation}

The set $L$ in the statement is precompact. Apply \cref{prop:local-covariance} to the two compact sets $K$ and $\overline L$. For all sufficiently large $i$, every $p\in D_i$ and every $g\in L$ satisfy
\[
 \Phi_i(g.p)|_K=(g\Phi_i(p))|_K.
\]
Combining this identity with \eqref{eq:agreement-K-deltaX}, we get
\begin{equation}\label{eq:root-close-lusin}
 \rho_X(\Phi_i(g.p),g\Phi_i(p))<\delta_X
 \qquad(p\in D_i,\ g\in L).
\end{equation}

Fix $g\in L$ and $p\in D_i$.
The factor identity holds $\mu_\alpha$-almost everywhere
for this fixed $g$, and $\Phi_i(p)$ has law $\mu_\alpha$
by \cref{prop:exact-marginal}.
Consequently,
\begin{equation}\label{eq:factor-identity-random}
	g\Psi_i(p)
	=
	g\pi(\Phi_i(p))
	=
	\pi(g\Phi_i(p))
	\qquad\text{almost surely}.
\end{equation}
If both configurations
\[
 \Phi_i(g.p)\quad\text{and}\quad g\Phi_i(p)
\]
belong to $C$, then \eqref{eq:root-close-lusin}, \eqref{eq:lusin-modulus}, and \eqref{eq:factor-identity-random} give
\begin{align*}
 \rho_Y(\Psi_i(g.p),g\Psi_i(p))
 =\rho_Y\bigl(\pi(\Phi_i(g.p)),\pi(g\Phi_i(p))\bigr)<\eta.
\end{align*}
For each fixed $p$ and $g$, the random variable $\Phi_i(g.p)$ has law $\mu_\alpha$ by \cref{prop:exact-marginal}. The variable $g\Phi_i(p)$ also has law $\mu_\alpha$, because $\Phi_i(p)\sim\mu_\alpha$ and $\mu_\alpha$ is $G$-invariant. They need not be independent, but the union bound is sufficient:
\[
 \mathbb P\bigl(\Phi_i(g.p)\notin C
 \text{ or }g\Phi_i(p)\notin C\bigr)
 \leq2\eta.
\]
Since $\rho_Y\leq1$, we get
\begin{equation}\label{eq:expected-core-defect}
 \mathbb E\,\rho_Y(\Psi_i(g.p),g\Psi_i(p))
 \leq \eta+2\eta=3\eta
 \qquad(p\in D_i,\ g\in L).
\end{equation}

For \(p\in D_i=M_i[W_i^4]\) and large \(i\), every \(g\in L\subset W_i\) acts at \(p\). Hence there is no contribution from the undefined part of the domain. On $M_i\setminus D_i$, the sum of the metric integrand on the defined part and the contribution from the undefined part of the domain is bounded by \(1\). Using Tonelli's theorem and \eqref{eq:expected-core-defect}, we get
\begin{align*}
 &\mathbb E\left[
 \frac1{\Haar(L)}\int_L
 \Delta_{\rho_Y,M_i}(g,\Psi_i)\,d\Haar(g)
 \right]\\
 &\quad\leq
 \frac1{\Haar(L)}\int_L
 \left(\int_{D_i}3\eta\,d\dbar_i(p)
 +\dbar_i(M_i\setminus D_i)\right)d\Haar(g)\\
 &\quad=3\eta d_i+(1-d_i)
 \leq3\eta+(1-d_i).
\end{align*}
Letting $i\to\infty$ and using $d_i\to1$ gives
\[
 \limsup_{i\to\infty}
 \mathbb E\left[
 \frac1{\Haar(L)}\int_L
 \Delta_{\rho_Y,M_i}(g,\Psi_i)\,d\Haar(g)
 \right]
 \leq3\eta.
\]
Since $\eta>0$ was arbitrary, the expectations tend to zero. Markov's inequality then implies convergence to zero in probability, as required.
\end{proof}

Combining averaged equivariance with empirical convergence gives the random microstates used in the packing argument. The comparison theorem will later give the same lower bound for the standard model spaces.

\begin{corollary}\label{cor:good-microstates}
Let $L\subset G$ be a precompact identity neighborhood, $\delta>0$, and let $\cO$ be a weak-star neighborhood of $\nu$. Then, as $i \to \infty$,
\[
 \mathbb P\left(
 \Psi_i\in\Mapavg(M_i,Y,\rho_Y:L,\delta,\cO)
 \right)\longrightarrow1.
\]
\end{corollary}

\begin{proof}
By \cref{prop:avg-equivariance}, the averaged equivariance condition holds with probability tending to one, and by \cref{prop:empirical-factor} the empirical measure lies in $\cO$ with probability tending to one.
\end{proof}

\section{Exponential separation in nontrivial factors}
Assume from now on that the factor is nontrivial. So $\nu$ is not a point mass.

\begin{lemma}\label{lem:separated-compacts}
There exist compact sets $C_0,C_1\subset Y$ and constants $p_0,p_1,\kappa>0$ such that
\[
 \nu(C_j)=p_j>0\quad(j=0,1),
 \qquad
 \inf\{\rho_Y(y_0,y_1):y_j\in C_j\}\geq\kappa.
\]
\end{lemma}

\begin{proof}
Because $\nu$ is not a point mass, its support contains two distinct points $y_0$ and $y_1$. Put
\[
 d:=\rho_Y(y_0,y_1)>0
\]
and choose $0<r<d/4$. The open balls $B_Y(y_j,r)$ have positive $\nu$-measure because $y_j\in\operatorname{supp}\nu$. By inner regularity of the Borel probability measure $\nu$ on the compact metric space $Y$, there are compact sets
\[
 C_j\subset B_Y(y_j,r)
\]
with $p_j:=\nu(C_j)>0$. If $z_j\in C_j$, then
\[
 \rho_Y(z_0,z_1)
 \geq\rho_Y(y_0,y_1)-\rho_Y(z_0,y_0)-\rho_Y(z_1,y_1)
 >d-2r>d/2.
\]
So the conclusion holds with $\kappa:=d/2$.
\end{proof}

Define
\[
 a:Y\to\{0,1,*\},\qquad
 a(y)=\begin{cases}0,&y\in C_0,\\1,&y\in C_1,\\ *,&\text{otherwise}.
 \end{cases}
\]

\begin{lemma}\label{lem:local-label}
For every $\eta>0$ there are a compact $K\subset G$ and a $K$-local Borel map
$
 b:\cN(G)\to\{0,1,*\}
$
such that
\begin{equation}\label{eq:label-error}
 \mu_\alpha\{b\neq a\circ\pi\}<\eta.
\end{equation}
Also, if $\eta<\min(p_0,p_1)/2$, then
\begin{equation}\label{eq:label-masses}
 r_j:=\mu_\alpha\{b=j\}\geq p_j-\eta>0,
 \qquad j=0,1.
\end{equation}
\end{lemma}

\begin{proof}
Let
\[
 A_0:=(a\circ\pi)^{-1}(0),\qquad
 A_1:=(a\circ\pi)^{-1}(1),\qquad
 A_*:=(a\circ\pi)^{-1}(*).
\]
These three Borel sets form a partition of $\cN(G)$. Choose an increasing compact exhaustion $(K_n)$ of $G$ and set
\[
\mathscr F_n
:=
\sigma\bigl(\omega\mapsto\omega|_{K_n}\bigr).
\]
Since $\bigvee_n\mathscr F_n=\cB(\cN(G))$, the martingale convergence theorem gives, for each $j\in\{0,1,*\}$, when $n \to \infty$,
\[
 u_{j,n}:=\mathbb E(\1_{A_j}\mid\mathscr F_n)
 \longrightarrow\1_{A_j}
 \quad\text{in }L^1(\mu_\alpha).
\]
By the Doob--Dynkin lemma, for each
$j\in\{0,1,*\}$ we may choose a Borel function
\[
\widetilde u_{j,n}:\cN(K_n)\to[0,1]
\]
such that
\[
u_{j,n}(\omega)
=
\widetilde u_{j,n}(\omega|_{K_n})
\]
for $\mu_\alpha$-almost every $\omega$.
After redefining on an $\mathscr F_n$-measurable null set,
we may assume that these representatives are Borel,
$K_n$-local, and satisfy
\[
u_{0,n}+u_{1,n}+u_{*,n}=1
\]
everywhere. The exceptional set is $\mathscr F_n$-measurable. Redefine the triple there to be $(1,0,0)$. This changes none of the $L^1$ classes, preserves Borel measurability and
$K_n$-locality, and makes the displayed identity hold at every configuration. Choose $n$ so large that
\begin{equation}\label{eq:partition-martingale-error}
 2\sum_{j\in\{0,1,*\}}
 \|u_{j,n}-\1_{A_j}\|_1<\eta.
\end{equation}
Define $b(\omega)$ to be an index at which the three numbers
$u_{0,n}(\omega),u_{1,n}(\omega),u_{*,n}(\omega)$ attain their maximum, using the fixed order $0<1<*$ to break ties. Then $b$ is Borel and $K_n$-local.

We estimate its error. If $\omega\in A_j$ but $b(\omega)\neq j$, then some other conditional probability is at least $u_{j,n}(\omega)$. Since the three probabilities sum to one, this forces $u_{j,n}(\omega)\leq1/2$. This gives
\[
 \1_{A_j\cap\{b\neq j\}}(\omega)
 \leq2\1_{A_j}(\omega)(1-u_{j,n}(\omega)).
\]
After integration and summation over the three labels,
\begin{align*}
 \mu_\alpha\{b\neq a\circ\pi\}
 &\leq2\sum_j\int_{A_j}(1-u_{j,n})\,d\mu_\alpha\\
 &\leq2\sum_j\|u_{j,n}-\1_{A_j}\|_1<\eta
\end{align*}
by \eqref{eq:partition-martingale-error}. This proves \eqref{eq:label-error} with $K:=K_n$.

For $j=0,1$, the symmetric difference estimate gives
\[
 \mu_\alpha\{b=j\}
 \geq\mu_\alpha(A_j)-\mu_\alpha\{b\neq a\circ\pi\}
 >p_j-\eta.
\]
If $\eta<\min(p_0,p_1)/2$, these quantities are positive, proving \eqref{eq:label-masses}.
\end{proof}

\begin{remark}
This is the measurable analogue of finite-coordinate approximation for
Bernoulli factors. No continuity of the factor map is assumed: the argument
uses only that compact restrictions generate the standard Borel structure of
the Poisson configuration space.
\end{remark}

Let
$
 q_*:=\frac{p_0p_1}{4}>0.
$
Choose $\eta>0$ so small that
\begin{equation}\label{eq:eta-choice}
 \eta<\frac12\min(p_0,p_1),
 \qquad
 2\sqrt\eta<\frac{q_*}{2}.
\end{equation}
Apply \cref{lem:local-label} with this $\eta$, and fix the resulting compact window $K$ and local label $b$. Then \eqref{eq:label-masses} gives $r_j\geq p_j/2$ and so $r_0r_1\geq q_*$.

Let $(\Phi_i',\Psi_i')$ be an independent copy of $(\Phi_i,\Psi_i)$, built from independent Poisson processes $\Xi_i',Z_i'$. Define $Z_i^{\mathrm{opp}}$ by
\begin{equation}\label{eq:Zi}
 Z_i^\mathrm{opp}:=\frac1{V_i}\int_{D_i}
 \1_{\{(b(\Phi_i(p)),b(\Phi_i'(p)))\in\{(0,1),(1,0)\}\}}
 \dd\vol(p).
\end{equation}

\begin{proposition}\label{prop:opposite-label}
For all sufficiently large $i$,
\[
 \mathbb E Z_i^\mathrm{opp}=2d_i r_0r_1\geq2d_i q_*.
\]
There exists $\gamma=\gamma(K,\alpha,q_*)>0$ such that
\begin{equation}\label{eq:opposite-tail}
 \mathbb P(Z_i^\mathrm{opp}<q_*)\leq2e^{-\gamma V_i}
\end{equation}
for all sufficiently large $i$.
\end{proposition}

\begin{proof}
Fix $i$ large enough that $K\subset W_i$. For $p\in D_i$, the $K$-locality of $b$ and the definition of the completed Poisson configuration imply that $b(\Phi_i(p))$ depends only on the restriction $\Xi_i|_{K.p}$; the exterior process contributes no point to the window $K$. The exact marginal property gives
\[
 b(\Phi_i(p))\sim b_*\mu_\alpha.
\]
The primed and unprimed constructions are independent, so the two labels at a fixed $p$ are independent. This gives
\begin{align*}
 \mathbb P\bigl((b(\Phi_i(p)),b(\Phi_i'(p)))
 \in\{(0,1),(1,0)\}\bigr)=r_0r_1+r_1r_0=2r_0r_1.
\end{align*}
Tonelli's theorem now gives
\[
 \mathbb E Z_i^{\mathrm{opp}}
 =\frac{\vol_{M_i}(D_i)}{V_i}\,2r_0r_1
 =2d_i r_0r_1\geq2d_i q_*.
\]

Define on two configurations
\[
 F(\omega,\omega')
 :=\1_{\{(b(\omega),b(\omega'))\in\{(0,1),(1,0)\}\}}.
\]
This function takes values in $[0,1]$ and is local in the window $K$ in each coordinate. The two-model part of \cref{prop:local-concentration} then gives, for the fixed number $t=q_*/2$, a constant $\gamma>0$ such that
\[
 \mathbb P\left(
 |Z_i^{\mathrm{opp}}-\mathbb E Z_i^{\mathrm{opp}}|>q_*/2
 \right)\leq2e^{-\gamma V_i}
\]
for all sufficiently large $i$. Since $d_i\to1$, we may also assume $d_i\geq3/4$, and then
\[
 \mathbb E Z_i^{\mathrm{opp}}\geq2d_i q_*\geq3q_*/2.
\]
So the event $Z_i^{\mathrm{opp}}<q_*$ implies
\[
 |Z_i^{\mathrm{opp}}-\mathbb E Z_i^{\mathrm{opp}}|>q_*/2,
\]
which proves \eqref{eq:opposite-tail}.
\end{proof}

Define
\[
 R_i:=\frac1{V_i}\int_{M_i}
 \1_{\{b(\Phi_i(p))\neq a(\Psi_i(p))\}}
 \dd\vol(p),
\]
and define $R_i'$ analogously.

\begin{lemma}[Approximation error]\label{lem:Ri}
One has
\[
 \mathbb E R_i<\eta,
 \qquad
 \mathbb P(R_i>\sqrt\eta)<\sqrt\eta.
\]
The same statements hold for $R_i'$.
\end{lemma}

\begin{proof}
For every fixed $p\in M_i$, the exact marginal property and the definition $\Psi_i(p)=\pi(\Phi_i(p))$ give
\begin{align*}
 &\mathbb P\bigl(b(\Phi_i(p))\neq a(\Psi_i(p))\bigr)\\
 &\qquad=\mathbb P\bigl(b(\Phi_i(p))\neq a(\pi(\Phi_i(p)))\bigr)\\
 &\qquad=\mu_\alpha\{b\neq a\circ\pi\}<\eta.
\end{align*}
Tonelli's theorem then gives
\[
 \mathbb E R_i
 =\frac1{V_i}\int_{M_i}
 \mathbb P\bigl(b(\Phi_i(p))\neq a(\Psi_i(p))\bigr)
 \,d\vol_{M_i}(p)<\eta.
\]
Since $R_i\geq0$, Markov's inequality gives
\[
 \mathbb P(R_i>\sqrt\eta)
 \leq\frac{\mathbb E R_i}{\sqrt\eta}<\sqrt\eta.
\]
The primed construction has the same law, so the same argument applies to $R_i'$.
\end{proof}

\begin{lemma}\label{lem:metric-label}
Suppose
$ Z_i^\mathrm{opp}\geq q_*$,
 $R_i,R_i'\leq\sqrt\eta$.
Then
\begin{equation}\label{eq:metric-separation}
 \rho_Y^{M_i}(\Psi_i,\Psi_i')\geq
 \kappa(q_*-2\sqrt\eta)\geq\frac{\kappa q_*}{2}.
\end{equation}
\end{lemma}

\begin{proof}
Let \(E_i\subset D_i\) be the set of points \(p\) for which the two local labels are opposite. Its normalized volume is $Z_i^\mathrm{opp}\geq q_*$. Discard those \(p\in E_i\) for which at least one local label differs from the corresponding actual label. The discarded set has normalized volume at most
\[
R_i+R_i'\le 2\sqrt{\eta}.
\]
Thus, for every remaining \(p\), the two actual labels are opposite. Hence one of the two corresponding points lies in \(C_0\) and the other lies in \(C_1\). Their $\rho_Y$-distance is at least $\kappa$. Integrating proves the first inequality; the second follows from \eqref{eq:eta-choice}.
\end{proof}

\section{Proof of the main theorem}

Let $\mathscr L_i=L^1(M_i,Y)$ be the space of measurable maps
modulo equality almost everywhere, equipped with
$\rho_Y^{M_i}$.
By \cref{lem:mapavg-borel}, this space is Polish and every
averaged model space
$\Mapavg(M_i,Y,\rho_Y:L,\delta,\cO)$ is Borel.

Choose a countable dense family $(s_n)$ in $\mathscr L_i$,
with measurable representatives.
For every $n$, the function
\[
\omega\longmapsto
\int_{M_i}
\rho_Y\bigl(\Psi_i(\omega,p),s_n(p)\bigr)\,d\dbar_i(p)
\]
is measurable by joint measurability and parameterized
integration.
These functions are the distances from $\Psi_i$ to the
points $s_n$, and such distances generate the Borel
$\sigma$-algebra of $\mathscr L_i$.
Hence $\Psi_i$ is an $\mathscr L_i$-valued random variable.
The random variable $R_i$ is measurable by the same
parameterized integration argument.
Thus all conditioning events used below are measurable.

Fix arbitrary microstate parameters: a precompact identity neighborhood $L\subset G$, $\delta>0$, and a weak-star neighborhood $\cO$ of $\nu$. Let
\[
 \mathscr M_i:=\Mapavg(M_i,Y,\rho_Y:L,\delta,\cO).
\]
Define the event
\[
 \mathcal G_i:=\{\Psi_i\in\mathscr M_i\}\cap\{R_i\leq\sqrt\eta\}.
\]
By \cref{cor:good-microstates,lem:Ri},
\begin{equation}\label{eq:good-probability}
 \liminf_i\mathbb P(\mathcal G_i)\geq1-\sqrt\eta.
\end{equation}
Set $c_0:=(1-\sqrt\eta)/2>0$. Then $\mathbb P(\mathcal G_i)\geq c_0$ for all sufficiently large $i$.
For all sufficiently large \(i\), define a probability measure \(Q_i\) by conditioning the law of \(\Psi_i\) on \(\mathcal G_i\); that is, 
\[
Q_i(A)
=
\mathbb P(\Psi_i\in A\mid \mathcal G_i)
=
\frac{\mathbb P(\{\Psi_i\in A\}\cap \mathcal G_i)}
{\mathbb P(\mathcal G_i)}
\]
for every Borel set \(A\subset L^1(M_i,Y)\).
Then
$
Q_i(\mathscr M_i)=1.
$

Set
$
 \eps_0:=\frac{\kappa q_*}{4}>0.
$ We have the following result.

\begin{proposition}\label{prop:collision}
If $\psi,\psi'$ are independent with law $Q_i$, then for all sufficiently large $i$,
\begin{equation}\label{eq:collision}
 \mathbb P\bigl(\rho_Y^{M_i}(\psi,\psi')\leq\eps_0\bigr)
 \leq C_0e^{-\gamma V_i},
\end{equation}
where $C_0=2c_0^{-2}$ and $\gamma>0$ is independent of $L,\delta,\cO$.
\end{proposition}

\begin{proof}
Under the unconditioned product law, if both copies satisfy $R_i,R_i'\leq\sqrt\eta$ and their distance is at most $\eps_0<\kappa q_*/2$, then \cref{lem:metric-label} implies $Z_i^\mathrm{opp}<q_*$. This gives
\[
 \mathbb P\bigl(
 \rho_Y^{M_i}(\Psi_i,\Psi_i')\leq\eps_0,
 \mathcal G_i,\mathcal G_i'
 \bigr)
 \leq\mathbb P(Z_i^\mathrm{opp}<q_*)
 \leq2e^{-\gamma V_i}.
\]
Divide by $\mathbb P(\mathcal G_i)^2\geq c_0^2$. So conditioning on the
good-microstate event changes the collision bound only by a fixed
multiplicative factor, and the exponential decay rate in $V_i$ is preserved.
\end{proof}

\begin{proposition}\label{prop:packing}
For all sufficiently large $i$,
\[
 \Sep_{\eps_0}(\mathscr M_i,\rho_Y^{M_i})
 \geq \left\lfloor\exp\left(\frac\gamma4V_i\right)\right\rfloor.
\]
\end{proposition}

\begin{proof}
Let
$
 N_i:=\left\lfloor e^{\gamma V_i/4}\right\rfloor
$
and sample $\psi_1,\dots,\psi_{N_i}$ independently from $Q_i$. By the union bound and \cref{prop:collision},
\[
 \mathbb P\left(\exists j<k:\rho_Y^{M_i}(\psi_j,\psi_k)\leq\eps_0\right)
 \leq \binom{N_i}{2}C_0e^{-\gamma V_i}
 \leq C_0e^{-\gamma V_i/2}
\]
for all sufficiently large $i$. The right-hand side tends to zero, so there exists an $(\rho_Y^{M_i},\eps_0)$-separated $N_i$-tuple in $\mathscr M_i$.
\end{proof}

Now we can prove  the main theorem.

\begin{proof}[Proof of \cref{thm:main}]
If $G$ is discrete, then $G$ is countably infinite.
Set
$
c_G:=\Haar(\{e\})>0.
$
Since $\Haar=c_G$ times counting measure, the coordinate map
$
\omega\longmapsto
\bigl(\omega(\{g\})\bigr)_{g\in G}
$
identifies the Poisson action, up to inversion of the coordinates,
with the Bernoulli action having base law
$
\operatorname{Pois}(\alpha c_G).
$

Moreover, each $M_i$ is finite and
$
V_i=c_G|M_i|.
$
Completing the partial translations on $M_i$ arbitrarily gives a
usual sofic approximation of $G$, since the defining relations hold
on subsets of relative cardinality tending to one.
Since $V_i=c_G|M_i|$, normalization by volume and normalization by
cardinality differ only by a fixed positive constant.
Kerr's theorem \cite{KerrCPE} therefore gives completely positive
entropy for the given approximation.

From now on, assume that $G$ is nondiscrete.	
Let $\pi:(\cN(G),\mu_\alpha)\to(Z,\zeta)$ be a nontrivial measurable factor, and choose a compact metrizable continuous model $(Y,\nu)$ of this factor. The constants $\eps_0>0$ and $\gamma>0$ constructed above depend on the factor, on the chosen finite-window approximation of one nontrivial factor label, and on $\alpha$, but they do not depend on the microstate parameters $L,\delta,\cO$.

For every such triple, \cref{prop:packing} gives
\[
 \limsup_{i\to\infty}\frac1{V_i}
 \log\Sep_{\eps_0}
 \bigl(\Mapavg(M_i,Y,\rho_Y:L,\delta,\cO),\rho_Y^{M_i}\bigr)
 \geq\frac\gamma4.
\]
Taking the infimum over $\cO$, $L$, and $\delta$ gives
\[
 h^{\mathrm{meas,avg}}_{\Sigma,\eps_0}(G,Y,\nu,\rho_Y)
 \geq\frac\gamma4>0.
\]
Because $\eps_0$ and $\gamma$ are independent of the microstate parameters $(L,\delta,\cO)$, the positive lower bound survives the defining infima.

By \cref{prop:avg-standard},
\[
 h^{\mathrm{meas}}_{\Sigma,\eps_0}(G,Y,\nu,\rho_Y)
 =h^{\mathrm{meas,avg}}_{\Sigma,\eps_0}(G,Y,\nu,\rho_Y)>0.
\]
Since the full entropy is the limit as the separation scale decreases,
\[
 h^{\mathrm{meas}}_\Sigma(G,Y,\nu)
 =\lim_{\eps\downarrow0}h^{\mathrm{meas}}_{\Sigma,\eps}(G,Y,\nu,\rho_Y)
 \geq h^{\mathrm{meas}}_{\Sigma,\eps_0}(G,Y,\nu,\rho_Y)>0.
\]
	By \cref{cor:singh-measure-identification}, this is
Singh's measure sofic entropy of the factor.  Singh's
measure-conjugacy invariance shows that it is independent of the chosen
compact model. Since the factor was arbitrary, the Poisson action has
completely positive measure sofic entropy relative to $\Sigma$.
\end{proof}

\section{Necessity of the noncompactness hypothesis}

We now show that the noncompactness assumption in \cref{thm:main} is essential for the theorem as stated.

\begin{proposition}\label{prop:compact-obstruction}
Let \(G\) be a nontrivial compact lcsc group. For every \(i\), set \(M_i=G\) and let \(G\) act on \(M_i\) by left translation:
$
g\cdot x=gx
$. Then $(M_i)$ is a sofic approximation. For every $\alpha>0$, the Poisson process of intensity $\alpha\Haar$ does not have completely positive measure sofic entropy with respect to this approximation.
\end{proposition}

\begin{proof}
The total-point map
$
 N(\omega)=\omega(G)
$
is a nontrivial invariant factor with law $\lambda=\operatorname{Pois}(\alpha\Haar(G))$ and trivial action. To construct it on a compact model, set
\[
 Y_0:=\{0\}\cup\{y_n:n\in\mathbb N_0\},
 \qquad y_n:=\frac1{n+1},
\]
with the Euclidean metric and the trivial $G$-action. Let $\iota(n)=y_n$ and $\widetilde\lambda=\iota_*\lambda$. Since \(\widetilde\lambda(\{0\})=0\), the map \(\iota\) gives a measure conjugacy between the count-valued factor and the compact model \((Y_0,\widetilde\lambda)\).

Choose distinct $a,b\in\mathbb N_0$ with $\lambda(a),\lambda(b)>0$ and put
$
 \kappa:=|y_a-y_b|>0.
$
The points $y_a,y_b$ are isolated in $Y_0$, so their singleton indicators are continuous. We can choose a weak-star neighborhood $\cO$ of $\widetilde\lambda$ such that every $\beta\in\cO$ satisfies
\[
 \beta(\{y_a\})>\lambda(a)/2,
 \qquad
 \beta(\{y_b\})>\lambda(b)/2.
\]
Let $\overline\Haar:=\Haar/\Haar(G)$ be normalized Haar measure.
Let $\varphi:G\to Y_0$ be measurable with
$
\varphi_*\overline\Haar\in\cO,
$
and write
$
\beta:=\varphi_*\overline\Haar.
$
Since $(g,p)\mapsto(gp,p)$ preserves
$\overline\Haar\times\overline\Haar$, Fubini gives
\begin{align*}
 \frac1{\Haar(G)}\int_G\rho_{Y_0}^{G}(\varphi\circ g,\varphi)\,d\Haar(g)
 &=\iint_{Y_0\times Y_0}|y-z|\,d\beta(y)d\beta(z)\\
 &\geq2\kappa\,\beta(\{y_a\})\beta(\{y_b\})\\
 &>\frac\kappa2\lambda(a)\lambda(b).
\end{align*}
The action on $Y_0$ is trivial, so $g\circ\varphi=\varphi$. Taking $L=G$ and any
$
 0<\delta<\frac\kappa2\lambda(a)\lambda(b)
$
shows that
\[
 \Mapavg(G,Y_0,\rho_{Y_0}:G,\delta,\cO)=\varnothing.
\]
Since $\Map\subseteq\Mapavg$, the corresponding standard model space is also empty. Thus the nontrivial total-point factor has entropy $-\infty$
with respect to this approximation, so the Poisson action does
not have completely positive measure sofic entropy.
\end{proof}

The proof of \cref{thm:main} also shows where noncompactness is used. By \cref{lem:volume-divergence}, noncompactness implies that \(V_i\to\infty\). This allows the local bounded-difference estimate to yield a packing estimate that grows exponentially with \(V_i\).

\medskip

\textit{Acknowledgments}.
I am grateful to my advisor, Professor Siming Tu, for introducing me to the subject and for his continued guidance, support, and encouragement. While preparing this work, I used ChatGPT Plus for assistance with mathematical discussions and editorial revisions. I take full responsibility for the correctness of all proofs in the paper.

\appendix
\section{Comparison and invariance of spatial entropy }\label{app:singh-comparison}

\subsection{Invariance of spatial entropy}

\begin{lemma}[Measurability of averaged model spaces]
	\label{lem:mapavg-borel}
	Let $M$ be a local $G$-space with
	$0<\vol_M(M)<\infty$.
	Let $G\curvearrowright Y$ be a continuous action on a compact
	metrizable space, and let $\rho\leq1$ be a compatible metric
	on $Y$.
	Then $L^1(M,Y)$, equipped with $\rho^M$, is Polish.
	The empirical-measure map
	\[
	\mathsf E_M:L^1(M,Y)\longrightarrow\Prob(Y),
	\qquad
	\mathsf E_M(\psi)=\psi_*\overline{\vol}_M,
	\]
	is continuous, and the map
	$
	(g,\psi)\longmapsto\Delta_{\rho,M}(g,\psi)
	$
	is Borel on $G\times L^1(M,Y)$.
	Consequently,
	\[
	\psi\longmapsto
	\frac1{\Haar(L)}
	\int_L\Delta_{\rho,M}(g,\psi)\,d\Haar(g)
	\]
	is Borel for every nonempty precompact open set $L\subset G$.
	In particular, for every precompact open identity neighborhood
	$L\subset G$, every $\delta>0$, and every weak-star open set
	$\cO\subseteq\Prob(Y)$, the space
$
	\Mapavg(M,Y,\rho:L,\delta,\cO)
$
	is a Borel subset of $L^1(M,Y)$.
\end{lemma}

\begin{proof}
	The space $L^1(M,Y)$ is Polish because $(M,\overline{\vol}_M)$ is a standard
	probability space and $Y$ is complete and separable. If $\psi_n\to\psi$ in
	$\rho^M$, then $\psi_n\to\psi$ in measure. Uniform continuity of every
	$f\in C(Y)$, followed by bounded convergence in measure, gives
	\[
	\int_M f(\psi_n(p))\,d\overline{\vol}_M(p)
	\longrightarrow
	\int_M f(\psi(p))\,d\overline{\vol}_M(p).
	\]
	So $\mathsf E_M$ is continuous for the weak-star topology on
	$\Prob(Y)$.
	
	Fix $y_*\in Y$. For $g\in G$ define the partial-composition operator
	$C_g:L^1(M,Y)\to L^1(M,Y)$ by
	\[
	(C_g\psi)(p):=
	\begin{cases}
		\psi(g.p),&p\in M[g],\\
		y_*,&p\notin M[g].
	\end{cases}
	\]
	This is well defined on almost-everywhere classes because partial
	translations preserve the canonical measure and so carry null sets to null
	sets. We claim that $(g,\psi)\mapsto C_g\psi$ is Borel. Choose a countable
	dense family $(s_n)$ in $L^1(M,Y)$. For each fixed $n$, the map
	$g\mapsto C_gs_n$ is Borel: the domain of the partial action is Borel, the
	partial action is Borel, and the distance from $C_gs_n$ to any fixed member of
	a countable dense family is a parameterized integral of a bounded Borel
	function. For each $k$, define
	\[
	n_k(\psi)
	:=
	\min\bigl\{
	n\geq1:
	\rho^M(s_n,\psi)<2^{-k}
	\bigr\}.
	\]
	The density of $(s_n)$ makes this index well defined.
	Since
	$
	\psi\longmapsto\rho^M(s_n,\psi)
	$
	is continuous for every $n$, the map $n_k$ is Borel.
	Moreover,
	\[
	\rho^M(s_{n_k(\psi)},\psi)<2^{-k}.
	\]
	Local measure preservation gives the contraction estimate
	\[
	\rho^M(C_gs_{n_k(\psi)},C_g\psi)
	\leq \rho^M(s_{n_k(\psi)},\psi)<2^{-k}.
	\]
	So $C_gs_{n_k(\psi)}\to C_g\psi$ in $L^1$, proving the claim as a
	pointwise limit of Borel maps.
	
	The map $(g,\psi)\mapsto g\psi$, where $(g\psi)(p)=g(\psi(p))$, is
	continuous from $G\times L^1(M,Y)$ to $L^1(M,Y)$ by joint continuity of the
	action and compactness of $Y$. Also
	$g\mapsto\1_{M[g]}$ is Borel as an $L^1(M)$-valued map, again by
	parameterized integration. This gives
	\begin{align*}
		\Delta_{\rho,M}(g,\psi)
		&=\int_M\1_{M[g]}(p)
		\rho\bigl((C_g\psi)(p),(g\psi)(p)\bigr)
		\,d\overline{\vol}_M(p)\\
		&\qquad+\overline{\vol}_M(M\setminus M[g])
	\end{align*}
	is Borel in $(g,\psi)$, since integration of the pointwise metric is
	continuous in the $L^1$ variables. A final application of parameterized integration over \(L\) gives the averaged equivariance error, and the claim about
	$\Mapavg$ follows from continuity of $\mathsf E_M$.
\end{proof}

\begin{proposition}[Metric independence]\label{prop:compact-entropy-invariance}
	The quantities
	\[
	h^{\mathrm{meas}}_{\Sigma}(G,Y,\nu,\rho)
	\quad\text{and}\quad
	h^{\mathrm{meas,avg}}_{\Sigma}(G,Y,\nu,\rho)
	\]
	are independent of the compatible metric $\rho$ on the fixed compact model $Y$.
\end{proposition}

\begin{proof}
	It is enough, by \cref{prop:avg-standard}, to prove the claim for the averaged entropy. Let $\rho$ and $\rho'$ be compatible metrics bounded by $1$. Fix a separation scale $r\in(0,1)$. By uniform equivalence on the compact space $Y$, choose $u\in(0,1)$ such that
	\begin{equation}\label{eq:separation-modulus}
		\rho'(y,z)<u\Longrightarrow \rho(y,z)<r/2.
	\end{equation}
	Put $t:=ur/4$. If $(\rho')^M(\psi,\psi')\leq t$, then Markov's
inequality shows that the set where $\rho'\geq u$ has normalized volume
at most $r/4$. On its complement, \eqref{eq:separation-modulus} gives
$\rho<r/2$. Since $\rho\leq1$,
\[
\rho^M(\psi,\psi')<\frac r2+\frac r4=\frac{3r}{4}<r.
\]
By contraposition, $\rho^M(\psi,\psi')>r$ implies
$(\rho')^M(\psi,\psi')>t$. So every
$(\rho^M,r)$-separated family is $(\rho'^M,t)$-separated, with the
strict inequality required by the definition of separation.
	
	Now fix an averaged-equivariance tolerance $\delta'>0$. Choose $a\in(0,\delta'/2)$ and then choose $s\in(0,1)$ such that
	\begin{equation}\label{eq:equivariance-modulus}
		\rho(y,z)<s\Longrightarrow \rho'(y,z)<a.
	\end{equation}
	Splitting the defined part of the domain into $\{\rho<s\}$ and its complement, and using $s\leq1$ for the undefined-domain penalty, gives
	\begin{equation}\label{eq:defect-comparison}
		\Delta_{\rho',M}(g,\psi)
		\leq a+s^{-1}\Delta_{\rho,M}(g,\psi).
	\end{equation}
	Set $\delta:=s(\delta'-a)>0$. Then
	\begin{equation}\label{eq:map-inclusion}
		\Mapavg(M,Y,\rho:L,\delta,\cO)
		\subseteq
		\Mapavg(M,Y,\rho':L,\delta',\cO).
	\end{equation}
	Combining separation with \eqref{eq:map-inclusion}, for every $L$, $\cO$, and $\delta'>0$ we get
	\begin{align*}
		&\limsup_i\frac1{V_i}
		\log\Sep_t\bigl(\Mapavg(M_i,Y,\rho':L,\delta',\cO),\rho'^{M_i}\bigr)\\
		&\qquad\geq
		\limsup_i\frac1{V_i}
		\log\Sep_r\bigl(\Mapavg(M_i,Y,\rho:L,\delta,\cO),\rho^{M_i}\bigr).
	\end{align*}
	Taking the infima over $\delta'$, $L$, and $\cO$, and then letting $r\downarrow0$, gives
	\[
	h^{\mathrm{meas,avg}}_{\Sigma}(G,Y,\nu,\rho')
	\geq h^{\mathrm{meas,avg}}_{\Sigma}(G,Y,\nu,\rho).
	\]
	Interchanging \(\rho\) and \(\rho'\) proves equality. By \cref{prop:avg-standard}, the same equality holds for the pointwise measure sofic entropy.
\end{proof}

\begin{proposition}[Topological-conjugacy invariance]\label{prop:top-conj-invariance}
	Let $G\curvearrowright(Y,\nu)$ and $G\curvearrowright(Y',\nu')$ be compact metrizable continuous pmp systems. Suppose that
	\[
	\Phi:Y\longrightarrow Y'
	\]
	is a $G$-equivariant homeomorphism with $\Phi_*\nu=\nu'$. Then
	\[
	h^{\mathrm{meas}}_\Sigma(G,Y,\nu)=h^{\mathrm{meas}}_\Sigma(G,Y',\nu')
	\]
	and
	\[
	h^{\mathrm{meas,avg}}_\Sigma(G,Y,\nu)
	=h^{\mathrm{meas,avg}}_\Sigma(G,Y',\nu').
	\]
	So the spatial formula is invariant under measure-preserving topological $G$-conjugacy.
\end{proposition}

\begin{proof}
	Let $\rho'$ be a compatible metric on $Y'$ bounded by $1$, and pull it back to the compatible metric
	\[
	\widetilde\rho(y,z):=\rho'(\Phi(y),\Phi(z))
	\qquad(y,z\in Y)
	\]
	on $Y$. For every finite-volume local $G$-space $M$, composition with $\Phi$ gives a bijection
	\[
	T_M:Y^M\longrightarrow (Y')^M,
	\qquad T_M(\psi):=\Phi\circ\psi.
	\]
	Because $\Phi$ is $G$-equivariant, for every $g\in G$ and every measurable $\psi:M\to Y$ we have
	\begin{align*}
		\Delta_{\rho',M}(g,T_M\psi)
		&=\int_{M[g]}\rho'\bigl(\Phi(\psi(g.p)),g\Phi(\psi(p))\bigr)
		\,d\overline{\vol}_M(p)
		+\overline{\vol}_M(M\setminus M[g])\\
		&=\int_{M[g]}\widetilde\rho\bigl(\psi(g.p),g\psi(p)\bigr)
		\,d\overline{\vol}_M(p)
		+\overline{\vol}_M(M\setminus M[g])\\
		&=\Delta_{\widetilde\rho,M}(g,\psi).
	\end{align*}
	Also,
	\[
	(T_M\psi)_*\overline{\vol}_M
	=\Phi_*\bigl(\psi_*\overline{\vol}_M\bigr),
	\]
	and for all $\psi,\psi':M\to Y$,
	\[
	(\rho')^M(T_M\psi,T_M\psi')
	=\widetilde\rho^M(\psi,\psi').
	\]
	
	The pushforward map
	\[
	\Phi_*:\operatorname{Prob}(Y)\longrightarrow\operatorname{Prob}(Y')
	\]
	is a homeomorphism for the weak-star topologies, with inverse $(\Phi^{-1})_*$. So if $\cO'$ is a weak-star neighborhood of $\nu'$ and
	\[
	\cO:=\Phi_*^{-1}(\cO'),
	\]
	then $\cO$ is a weak-star neighborhood of $\nu$, and for every precompact open identity neighborhood $L\subset G$ and every $\delta>0$,
	\begin{align*}
		T_M\bigl(\Map(M,Y,\widetilde\rho:L,\delta,\cO)\bigr)
		&=\Map(M,Y',\rho':L,\delta,\cO'),\\
		T_M\bigl(\Mapavg(M,Y,\widetilde\rho:L,\delta,\cO)\bigr)
		&=\Mapavg(M,Y',\rho':L,\delta,\cO').
	\end{align*}
	These bijections are isometries for the normalized integral metrics. So for every $\eps>0$, every $i$, and the corresponding neighborhoods $\cO,\cO'$,
	\begin{align*}
		\Sep_\eps\bigl(\Map(M_i,Y,\widetilde\rho:L,\delta,\cO),
		\widetilde\rho^{M_i}\bigr)=
		\Sep_\eps\bigl(\Map(M_i,Y',\rho':L,\delta,\cO'),
		(\rho')^{M_i}\bigr),
	\end{align*}
	and the same equality holds for the averaged model spaces. Since $\Phi_*$ bijects the weak-star neighborhood filters of $\nu$ and $\nu'$, taking the infima in the definitions gives
	\[
	h^{\mathrm{meas}}_{\Sigma,\eps}(G,Y,\nu,\widetilde\rho)
	=h^{\mathrm{meas}}_{\Sigma,\eps}(G,Y',\nu',\rho')
	\]
	for every $\eps>0$, and also for $h^{\mathrm{meas,avg}}_{\Sigma,\eps}$. Letting $\eps\downarrow0$ gives equality of the corresponding entropies for the metrics $\widetilde\rho$ and $\rho'$. Finally, metric independence from \cref{prop:compact-entropy-invariance} allows $\widetilde\rho$ and $\rho'$ to be replaced by arbitrary compatible metrics on $Y$ and $Y'$, respectively.
\end{proof}
	\begin{remark}\label{rem:top-vs-measure-conjugacy}

		The previous proposition compares compact models only when they are
		topologically conjugate. Two compact models of the same standard pmp action
		need only be measure conjugate. The required compact-model independence is
		proved below by identifying the local formula with Singh's
		measure-conjugacy invariant entropy.

\end{remark}

\subsection{Comparison with Singh’s pseudometric entropy}
In this subsection we verify, in the common compact-model setting, that the
spatial measure entropy of \cref{def:spatial-entropy} agrees with the
pseudometric microstate entropy introduced by Singh in
\cite[Definitions~9--11]{SinghThesis}. The two formulations differ in three ways. First, the present paper uses local $G$-spaces, so the translations are only
partially defined. Singh uses globally defined measurable maps. Second, the
present paper uses an $L^1$ model pseudometric, while Singh
uses the corresponding $L^2$ pseudometric. Third, the empirical-measure 
condition is expressed here by a weak-star neighborhood and in Singh's definition
by finitely many continuous test functions. We show that none of these differences affects the entropy.

Let $G\curvearrowright (X,\mu)$ be a jointly continuous
pmp action on a compact metrizable space  and $\Sigma=(M_i)_{i\geq 1}$ be a sofic approximation
in the sense of Definition \ref{def:sofic-approximation}, and let \(\rho\le1\) be a compatible metric on \(X\).

	Locally compact sofic groups are unimodular. So the right Haar measure
used to define the canonical measures on the local $G$-spaces is also a left
Haar measure and may be used in Singh's Definition~4 without changing the
normalization. The completion constructed below depends only on the original
approximation $\Sigma$, not on the action, factor, or compact model.

\subsubsection{From local \texorpdfstring{$G$}{G}-spaces to
	sofic approximations in Singh's sense}

Write
$
\vol_i:=\vol_{M_i},
$ $
\overline{\vol}_i
:=
\frac{\vol_i}{\vol_i(M_i)}.
$
For $g\in G$, put
\[
M_i[g]
:=
\{p\in M_i:(g,p)\text{ lies in the domain of the partial action}\}.
\]
Complete the partial action by setting
\begin{equation}\label{eq:completion}
	\widetilde L_i(g,p)
	:=
	\begin{cases}
		g.p, & p\in M_i[g],\\
		p,   & p\notin M_i[g].
	\end{cases}
\end{equation}

\begin{lemma}\label{lem:completion-is-singh}
	The sequence
$
	\widetilde\Sigma
	:=(M_i,\vol_i,\widetilde L_i)_{i\geq 1}
$
	is a sofic approximation in the sense of Singh.
\end{lemma}

\begin{proof}
	Fix a precompact open neighborhood $U$ of the identity. Let $M_i[U]$ be as in
	\cref{def:local-sofic-model} and put
	\[
	\beta_i(U):=\overline\vol_i(M_i\setminus M_i[U]).
	\]
	The definition of a sofic approximation implies
	\begin{equation}\label{eq:beta-to-zero}
		\beta_i(U)\longrightarrow 0.
	\end{equation}
	In fact, for all sufficiently large $i$, one has $U\subset U_i$ and so
	$M_i[U_i]\subset M_i[U]$.
	
	If $p\in M_i[U]$, then $g\mapsto g.p$ is a homeomorphism from $U$ onto an open
	neighborhood of $p$, and
	\[
	g.h.p=gh.p
	\qquad(g,h,gh\in U).
	\]
	Also, the defining property of the canonical measure gives
	\[
	\vol_i(\{g.p:g\in A\})=\Haar(A)
	\]
	for every Borel set $A\subset U$. So the restriction of
	$g\mapsto\widetilde L_i(g,p)=g.p$ to $U$ is a measure-space isomorphism onto its
	image, it sends the identity to $p$, and it satisfies the required local
	multiplication law. These are exactly conditions~(1)--(3) of
	\cite[Definition~4]{SinghThesis}. Since $M_i[U]$ has normalized measure tending to one by
	\eqref{eq:beta-to-zero}, the sequence $\widetilde\Sigma$ satisfies Singh's
	sofic-approximation axioms.
\end{proof}

\subsubsection{The two microstate spaces}

For measurable maps $\phi,\psi:M_i\to X$, define
\begin{align*}
	\rho_{1,i}(\phi,\psi)
	&:=\int_{M_i}\rho(\phi(p),\psi(p))\,d\overline\vol_i(p),\\
	\rho_{2,i}(\phi,\psi)
	&:=\left(\int_{M_i}\rho(\phi(p),\psi(p))^2
	\,d\overline\vol_i(p)\right)^{1/2}.
\end{align*}
For $g\in G$, define the Bowen equivariance error
\begin{equation}\label{eq:bowen-error}
E^B_i(\phi,g)
	:=\overline\vol_i(M_i\setminus M_i[g])
	+\int_{M_i[g]}
	\rho\bigl(\phi(g.p),g\phi(p)\bigr)\,d\overline\vol_i(p),
\end{equation}
and the completed Singh equivariance error
\begin{equation}\label{eq:singh-error}
E^S_i(\phi,g)
	:=\left(\int_{M_i}
	\rho\bigl(\phi(\widetilde L_i(g,p)),g\phi(p)\bigr)^2
	\,d\overline\vol_i(p)\right)^{1/2}.
\end{equation}
So $E^B_i(\phi,g)$ is exactly the error appearing in \eqref{eq:defect}.

If $\mathcal F\subset C(X)$ is finite and $\eta>0$, let
\begin{equation}\label{eq:weak-star-basic}
	\cO(\mathcal F,\eta)
	:=\left\{\nu\in\Prob(X):
	\left|\int f\,d\nu-\int f\,d\mu\right|<\eta
	\text{ for every }f\in\mathcal F\right\}.
\end{equation}
These sets form a neighborhood basis of $\mu$ for the weak-star topology on
$\Prob(X)$.

For a precompact open neighborhood $U\subset G$ and $\delta>0$, write
\begin{align*}
	\Map_B(i;U,\delta,\mathcal F,\eta)
	:=\{\phi:M_i\to X:\;&E^B_i(\phi,g)<\delta\text{ for all }g\in U,\\
	&\phi_*\overline\vol_i\in\cO(\mathcal F,\eta)\},
\end{align*}
and
\begin{align*}
	\Map_S(i;U,\delta,\mathcal F,\eta)
	:=\{\phi:M_i\to X:\;&E^S_i(\phi,g)<\delta\text{ for all }g\in U,\\
	&\phi_*\overline\vol_i\in\cO(\mathcal F,\eta)\}.
\end{align*}
The first is the pointwise microstate space from Section~2, restricted to the basic
weak-star neighborhood \eqref{eq:weak-star-basic}. The second is Singh's spatial
microstate space, except that we temporarily allow separate tolerances for
equivariance and empirical distribution. This does not alter Singh's entropy:
the diagonal family formed by setting both tolerances equal is cofinal, because
for any $\delta,\eta>0$ one may replace them by $\min\{\delta,\eta\}$.

\begin{lemma}\label{lem:error-comparison}
	For every precompact open $U\subset G$, every $g\in U$, every $i$, and every
	measurable $\phi:M_i\to X$,
	\begin{equation}\label{eq:error-comparison}
		E^S_i(\phi,g)^2\le E^B_i(\phi,g)
		\quad\text{and}\quad
		E^B_i(\phi,g)\le E^S_i(\phi,g)+\beta_i(U).
	\end{equation}
	These inequalities give, for every $0<\delta<1$,
	\begin{equation}\label{eq:map-inclusions}
		\Map_B(i;U,\delta^2,\mathcal F,\eta)
		\subseteq
		\Map_S(i;U,\delta,\mathcal F,\eta),
	\end{equation}
	and, for all sufficiently large $i$,
	\begin{equation}\label{eq:map-inclusions-reverse}
		\Map_S(i;U,\delta/2,\mathcal F,\eta)
		\subseteq
		\Map_B(i;U,\delta,\mathcal F,\eta).
	\end{equation}
\end{lemma}

\begin{proof}
	Fix $g\in U$. Since $M_i[U]\subset M_i[g]$, one has
	\begin{equation}\label{eq:missing-mass}
		\overline\vol_i(M_i\setminus M_i[g])\le\beta_i(U).
	\end{equation}
	Put
	\[
	a(p):=\rho(\phi(g.p),g\phi(p))\quad(p\in M_i[g])
	\]
	and
	\[
	c(p):=\rho(\phi(\widetilde L_i(g,p)),g\phi(p))
	\quad(p\in M_i\setminus M_i[g]).
	\]
	Because $0\le a,c\le1$ and $\widetilde L_i(g,p)=g.p$ on $M_i[g]$,
	\begin{align*}
		E^S_i(\phi,g)^2
		&=\int_{M_i[g]} a(p)^2\,d\overline\vol_i(p)
		+\int_{M_i\setminus M_i[g]}c(p)^2\,d\overline\vol_i(p)\\
		&\le\int_{M_i[g]} a(p)\,d\overline\vol_i(p)
		+\overline\vol_i(M_i\setminus M_i[g])
		=E^B_i(\phi,g).
	\end{align*}
	But by Cauchy--Schwarz and \eqref{eq:missing-mass},
	\begin{align*}
		E^B_i(\phi,g)
		&\le\beta_i(U)+\int_{M_i[g]} a(p)\,d\overline\vol_i(p)\\
		&\le\beta_i(U)
		+\left(\int_{M_i[g]} a(p)^2\,d\overline\vol_i(p)\right)^{1/2}
		\le\beta_i(U)+E^S_i(\phi,g).
	\end{align*}
	This proves \eqref{eq:error-comparison}. The inclusion
	\eqref{eq:map-inclusions} follows directly. By
	\eqref{eq:beta-to-zero}, for all sufficiently large $i$ one has
	$\beta_i(U)<\delta/2$, and then \eqref{eq:map-inclusions-reverse} follows from
	the second inequality in \eqref{eq:error-comparison}.
\end{proof}

\begin{lemma}\label{lem:l1-l2-separation}
	For all measurable $\phi,\psi:M_i\to X$,
	\begin{equation}\label{eq:l1-l2-model}
		\rho_{1,i}(\phi,\psi)
		\le\rho_{2,i}(\phi,\psi)
		\le\rho_{1,i}(\phi,\psi)^{1/2}.
	\end{equation}
	These inequalities give, for every $A\subset\Map(M_i,X)$ and every $0<\eps<1$,
	\begin{equation}\label{eq:sep-comparison}
		\Sep_\eps(A,\rho_{1,i})
		\le\Sep_\eps(A,\rho_{2,i})
		\le\Sep_{\eps^2}(A,\rho_{1,i}).
	\end{equation}
\end{lemma}

\begin{proof}
	Apply $\|f\|_1\le\|f\|_2$ and $f^2\le f$ to the function
	$f(p)=\rho(\phi(p),\psi(p))\in[0,1]$. The separation-number inequalities are
	immediate from \eqref{eq:l1-l2-model}.
\end{proof}

\subsubsection{Equality of the entropies}

For $0<\eps<1$, define
\begin{align*}
	b_\eps
	:=\inf_{\mathcal F}\inf_{\eta>0}\inf_U\inf_{\delta>0}
	\limsup_{i\to\infty}\frac{1}{\vol_i(M_i)}
	\log\Sep_\eps
	\bigl(\Map_B(i;U,\delta,\mathcal F,\eta),\rho_{1,i}\bigr),
\end{align*}
where $\mathcal F$ ranges over finite subsets of $C(X)$ and $U$ over precompact
open neighborhoods of the identity. Because the sets
$\cO(\mathcal F,\eta)$ form a weak-star neighborhood basis,
\begin{equation}\label{eq:bowen-b-epsilon}
	h^{\mathrm{meas}}_\Sigma(G,X,\mu,\rho)=\sup_{0<\eps<1} b_\eps.
\end{equation}
In the same way, let
\begin{align*}
	s_\eps
	:=\inf_{\mathcal F}\inf_{\eta>0}\inf_U\inf_{\delta>0}
	\limsup_{i\to\infty}\frac{1}{\vol_i(M_i)}
	\log\Sep_\eps
	\bigl(\Map_S(i;U,\delta,\mathcal F,\eta),\rho_{2,i}\bigr).
\end{align*}
By the cofinality observation above,
\[
\sup_{0<\eps<1}s_\eps
=h_{\widetilde\Sigma}^{\mathrm{Singh,pseudo}}(\mu,\rho),
\]
where the right-hand side is Singh's pseudometric entropy computed with the
completed approximation $\widetilde\Sigma$.

	\begin{theorem}[Comparison with Singh's pseudometric entropy]\label{thm:identification-bowen}
	Under the assumptions of this subsection,
	\begin{equation}\label{eq:main-equivalence}
		h^{\mathrm{meas}}_\Sigma(G,X,\mu,\rho)
		=h_{\widetilde\Sigma}^{\mathrm{Singh,pseudo}}(\mu,\rho).
	\end{equation}
	Here $\widetilde\Sigma$ is the Singh approximation built by completing
	the same local approximation $\Sigma$. The common value is independent of
	how the partial action is completed outside its domain, and the completed
	approximation depends only on $\Sigma$, not on the action or compact model.
\end{theorem}

\begin{proof}
	Fix $0<\eps<1$. By \eqref{eq:map-inclusions} and the first inequality in
	\eqref{eq:sep-comparison}, for every choice of
	$U,\mathcal F,\eta,$ and $0<\delta<1$,
	\begin{align*}
		\Sep_\eps
		\bigl(\Map_B(i;U,\delta^2,\mathcal F,\eta),\rho_{1,i}\bigr)\le
		\Sep_\eps
		\bigl(\Map_S(i;U,\delta,\mathcal F,\eta),\rho_{2,i}\bigr).
	\end{align*}
	Taking normalized logarithms, the limsup, and then the indicated infima gives
	\begin{equation}\label{eq:b-less-s}
		b_\eps\le s_\eps.
	\end{equation}
	Here replacing $\delta$ by $\delta^2$ does not affect the infimum, since
	$\{\delta^2:0<\delta<1\}$ is cofinal at zero.

	For the reverse direction, by \eqref{eq:map-inclusions-reverse}, for all sufficiently large $i$,
	\[
	\Map_S(i;U,\delta/2,\mathcal F,\eta)
	\subseteq\Map_B(i;U,\delta,\mathcal F,\eta).
	\]
	Using the second inequality in \eqref{eq:sep-comparison},
	\begin{align*}
		\Sep_\eps
		\bigl(\Map_S(i;U,\delta/2,\mathcal F,\eta),\rho_{2,i}\bigr)
		\le
		\Sep_{\eps^2}
		\bigl(\Map_B(i;U,\delta,\mathcal F,\eta),\rho_{1,i}\bigr)
	\end{align*}
	for all sufficiently large $i$. This gives
	\begin{equation}\label{eq:s-less-b}
		s_\eps\le b_{\eps^2}.
	\end{equation}
	Combining \eqref{eq:b-less-s} and \eqref{eq:s-less-b},
	\[
	b_\eps\le s_\eps\le b_{\eps^2}.
	\]
	Taking the supremum over $0<\eps<1$ gives
	\[
	\sup_{0<\eps<1}b_\eps
	=\sup_{0<\eps<1}s_\eps,
	\]
	because $\eps\mapsto\eps^2$ maps $(0,1)$ onto $(0,1)$. This proves
	\eqref{eq:main-equivalence}. The argument used no property of the completion
	outside the estimate $0\le c\le1$; so every Borel completion gives the same
	entropy.
\end{proof}

	\begin{corollary}[Identification with Singh's measure entropy]
	\label{cor:singh-measure-identification}
	Assume in addition that $G$ is nondiscrete. In the notation of Singh's
	thesis, for every dynamically generating sequence $P$, every countable set
	$\Gamma\subset G$ which generates $P$ in Singh's sense, and every
	constant $C\geq1$,
	\begin{align*}
		h^{\mathrm{meas}}_\Sigma(G,X,\mu,\rho)
		&=h_{\widetilde\Sigma}^{\mathrm{Singh,pseudo}}(\mu,\rho)\\
		&=h_{\widetilde\Sigma}^{\mathrm{meas,Singh}}
		\bigl(G\curvearrowright(X,\mu)\bigr)\\
		&=h_{\operatorname{Hom}_C}(P,\Gamma)
		=h_{\operatorname{UP}_C}(P,\Gamma).
	\end{align*}
	So the local-$G$-space formula is invariant under measure
	conjugacy and is independent of the compact metrizable continuous model and
	compatible metric used to represent the underlying standard pmp action.
\end{corollary}

\begin{proof}
	The first equality is \cref{thm:identification-bowen}. Singh's
	Theorem~3.4.2 identifies the unrestricted and $L^2$-bounded pseudometric
	formulas for nondiscrete groups. Proposition~3.4.3 identifies this pseudometric entropy with the unital-homomorphism formula, while Theorem~3.3.2 and
	Definition~11 identify the homomorphism and unital-positive-map formulas and
	define their common value to be the measure entropy of the action. The
	independence of the generating data proved in Chapter~3 makes this common
	value a measure-conjugacy invariant; see
	\cite[Theorems~3.2.3, 3.3.2, 3.4.2, Proposition~3.4.3, and
	Definition~11]{SinghThesis}.
\end{proof}

\begin{remark}\label{rem:scope-comparison}

	\Cref{thm:identification-bowen} compares the local formula for
	$\Sigma$ with Singh's formula for the particular completion
	$\widetilde\Sigma$ built from $\Sigma$; it does not compare entropies
	attached to two unrelated sofic approximation sequences.  We use compact models in order to work directly in Singh's spatial
	setting. Since every standard pmp action of an lcsc group has a compact
	metrizable continuous model, no comparison with noncompact models is needed
	for the present theorem.
\end{remark}

\end{document}